\documentclass[a4paper,reqno,final]{amsart}

\usepackage{amsmath,amssymb,amsthm}
\usepackage[foot]{amsaddr}
\usepackage{graphicx}
\usepackage{cite}
\usepackage{xcolor}
\usepackage{mathtools}
\usepackage{enumerate}
\usepackage{nicefrac}
\usepackage{bm}
\usepackage{algpseudocode}
\usepackage{booktabs}

\title[Galerkin Time Stepping Methods for Nonlinear IVP]{Continuous and Discontinuous Galerkin Time Stepping Methods for Nonlinear Initial Value Problems with Application to Finite Time Blow-Up}
\author{B\"arbel Holm}
\address{Department for Computational Science and Technology,
School of Computer Science and Communication,
KTH Royal Institute of Technology,
SE-100 44 Stockholm, Sweden}
\email{barbel@kth.se}

\author{Thomas P.~Wihler}
\address{Mathematisches Institut, Universit\"at Bern, Sidlerstr.~5, CH-3012 Bern, Switzerland}
\email{wihler@math.unibe.ch}

\thanks{The authors acknowledge the support of the Swiss National Science Foundation (SNF), Grant No.~200021-162990.}

\def\F{\mathcal F}
\def\G{\mathcal G}
\def\H{\mathsf T^{\dG}}
\def\HcG{\mathsf T^{\cG}}

\def\M{\mathcal M}

\def\dG{\mathsf{dG}}
\def\cG{\mathsf{cG}}

\def\Up{U^+}
\def\Um{U^-}

\newcommand{\jump}[1]{\lbrack\!\lbrack#1\rbrack\!\rbrack} %[[ ]]
\newcommand{\norm}[1]{\left\lVert#1\right\rVert} % || ||
\newcommand{\dif}{\,\text{d}}
\renewcommand{\L}{\mathsf{L}}
\newcommand{\NN}[1]{\left\|#1\right\|}

\renewcommand{\phi}{\varphi}
\newcommand{\dd}{\mathsf{d}}
\renewcommand{\P}{\Upsilon^{r_m}_{m}}
\newcommand{\cF}{c_{\F}}

\newtheorem{algorithm}{Algorithm}

\newtheorem{Remark}{Remark}
\newenvironment{remark}{\begin{Remark}\rm}{\end{Remark}}
\newtheorem{theorem}{Theorem}
\newtheorem{proposition}{Proposition}

\newtheorem{lemma}{Lemma}

\numberwithin{equation}{section}

\begin{document}

\begin{abstract}
We consider continuous and discontinuous Galerkin time stepping methods of arbitrary order as applied to nonlinear initial value problems in real Hilbert spaces. Our only assumption is that the nonlinearities are continuous; in particular, we include the case of unbounded nonlinear operators. Specifically, we develop new techniques to prove general Peano-type existence results for discrete solutions. In particular, our results show that the existence of solutions is independent of the local approximation order, and only requires the local time steps to be sufficiently small (independent of the polynomial degree). The uniqueness of (local) solutions is addressed as well. In addition, our theory is applied to finite time blow-up problems with nonlinearities of algebraic growth. For such problems we develop a time step selection algorithm for the purpose of numerically computing the blow-up time, and provide a convergence result.
\end{abstract}



\maketitle

%%%%%%%%%%%%%%%%%%%%%%%%%%%%%%%%%%%%%%%%%%%%%%%%%%%%%%%%%%%%%%%%%%%%%%
%%%%%%%%%%%%%%%%%%%%%%%%%%%%%%%%%%%%%%%%%%%%%%%%%%%%%%%%%%%%%%%%%%%%%%
\section{Introduction}
%%%%%%%%%%%%%%%%%%%%%%%%%%%%%%%%%%%%%%%%%%%%%%%%%%%%%%%%%%%%%%%%%%%%%%
%%%%%%%%%%%%%%%%%%%%%%%%%%%%%%%%%%%%%%%%%%%%%%%%%%%%%%%%%%%%%%%%%%%%%%

In this paper we focus on continuous Galerkin (cG) as well as on discontinuous Galerkin (dG) time stepping discretizations (of any order) as applied to abstract initial value problems of the form
\begin{align}
  u'(t) &= \F(t,u(t)), \quad t\in(0,T),\qquad\quad u(0) = u_0.\label{eq:1}
\end{align}
Here, $u:\,(0,T)\to H$, for some~$T>0$, is an unknown solution, with values in a real Hilbert space~$H$ (with inner product denoted by~$(\cdot,\cdot)_H$ and induced norm~$\|\cdot\|_H$).  The initial value $u_0 \in H$ prescribes the solution~$u$ at the start, $t=0$, and~$\F:\,[0,T]\times H\to H$ is a possibly nonlinear, {\em continuous} operator. We emphasize that we include, for instance, the case of~$\F$ being (continuous and nonlinear and) {\em unbounded} in the sense that
\begin{equation}\label{eq:unbounded}
\frac{\|\F(t,x)\|_H}{\|x\|_H}\to\infty\text{ as }\|x\|_H\to\infty,\qquad 0\le t\le T.
\end{equation}
In the sequel, we will usually omit to explicitly write the dependence on the first argument~$t$.

%We will look at different types of nonlinearities: 
%
%
% In addition, we will especially focus on nonlinearities which satisfy the growth condition
%\begin{equation}
%  \NN{\F(u)}_{H} \le c \NN{u}_H^\beta\qquad\forall u\in H,
%\label{condF}
%\end{equation}
%for some constants~$\beta\ge 1, c>0$. In the latter case, we notice that the problem~\eqref{eq:1}--\eqref{eq:2} might neither be monotone nor coercive. 

For~$H=\mathbb{R}^N$ and continuous nonlinearities~$\F$, the well-known Peano Theorem (see, e.g., \cite{Teschl}) guarantees the existence of $C^1$-solutions~$u$ of~\eqref{eq:1} within some limited time range, $t\in (0,T_\infty)$, for some~$T_\infty>0$. Generalizations to problems in Banach spaces are available as well; see, e.g., \cite{FrLe93}. Notice that the existence interval for  solutions may be arbitrarily small even for smooth~$\F$: For instance, solutions of~\eqref{eq:1} may become unbounded in finite time, i.e.,
\[
\|u(t)\|_H<\infty\text{ for }0<t<T_\infty,\qquad \lim_{t\nearrow T_\infty}\|u(t)\|_H = \infty.
\]
This effect is commonly termed (finite-time) \emph{blow-up}.

\subsection*{Galerkin Time Stepping} Galerkin-type time stepping methods for initial-value problems are based on weak formulations. For both the cG and the dG time stepping schemes, the test spaces consist of polynomials that are discontinuous at the time nodes. In this way, the discrete Galerkin formulations decouple into local problems on each time step, and the discretizations can therefore be understood as implicit one-step schemes. Galerkin time stepping methods have been analyzed for ordinary differential equations (ODEs), e.g.,
in~\cite{rannacher,DeDu86,DelfourHagerTrochu81,Estep95,EstepFrench94,johnson}. 

A key feature of Galerkin time stepping methods is their great flexibility with respect to the size of the time steps and the local approximation orders, thereby naturally leading to an $hp$-version Galerkin framework. The $hp$-versions of the cG and dG time stepping schemes were introduced and analyzed in the works~\cite{SchoetzauSchwabDGODE,SchoetzauSchwab00,ScWi10,Wihler05}. In particular, in the articles~\cite{SchoetzauSchwabDGODE,Wihler05}, which focus on ordinary initial value problems with uniform Lipschitz nonlinearities, the use of the contraction mapping theorem made it possible to prove existence and uniqueness results for discrete Galerkin solutions, which are independent of the local approximation orders. We emphasize that the $hp$-approach is well-known for its ability to approximate smooth solutions with possible local singularities at high algebraic or even exponential rates of convergence; see, e.g., \cite{BrunnerSchoetzau06,SchoetzauSchwab00,SchoetzauSchwab01,Gerdes} for the numerical approximation of problems with start-up singularities. 

\subsection*{Results} The goal of the current paper is to extend the existence results on $hp$-type Galerkin time stepping schemes for initial value problems with Lipschitz-type nonlinearities in~\cite{SchoetzauSchwabDGODE,Wihler05} to problems with nonlinearities which are merely continuous. We emphasize that this generalization is substantial; indeed, it covers, for example, the case of {\em unbounded} nonlinearities as in~\eqref{eq:unbounded}. We will develop a new technique which is based on writing the weak Galerkin formulations in strong form along the lines of~\cite{AkMaNo09,ScWi10}. Subsequently, suitable fixed-point forms will be derived. In the context of the cG method, this is accomplished within an integral equation framework. For the dG scheme, matters are more sophisticated, and a careful investigation of the discrete time derivative operator, which involves a lifting operator from~\cite{ScWi10}, is required on the local polynomial approximation space; this operator turns out to be an isomorphism on the underlying polynomial spaces (with a continuity constant of the inverse operator that is independent of the local polynomial degrees) and allows to transform the strong dG form into a fixed point equation. For both the cG and the dG schemes the application of Brower's fixed point theorem yields the existence of discrete solutions; see Theorem~\ref{thm:cGdG}. In particular, as in the case of Lipschitz continuous nonlinearities~\cite{AkMaNo09,ScWi10}, the existence results {\em do not depend on the local polynomial degrees}, and only require the local time steps to be sufficiently small. In this sense, our theory constitutes a \emph{discrete version of Peano's Theorem}. Furthermore, employing a contraction argument along the lines of the approach presented in~\cite{BandleBrunner:94}, we show that the local Galerkin formulations are uniquely solvable (within a certain range); cf.~Theorem~\ref{thm:uniqueness}.

In addition, we apply our general theory to initial value problems with nonlinearities of algebraic growth, i.e., $\F(t,u)\sim\alpha\|u\|_H^\beta$, with~$\alpha>0$, $\beta>1$, and for a given range of~$t$; in this case, the initial value problem~\eqref{eq:1} features a solution that blows up in a finite time~$T_\infty$. We will show that a careful selection of locally varying time steps in the cG and dG time stepping schemes results in discrete solutions that blow up as well; in this context, we mention the paper~\cite{StuartFloater:90} which illustrates the importance of variable step size selection. More precisely, following some ideas from~\cite{nakagawa:75}, we derive an analysis which allows to choose the local time steps \emph{a posteriori} as the time marching process is moving forward. We develop a time step selection algorithm which guarantees the existence and uniqueness of local solutions, and provides a numerical approximation of the exact blow-up time. Moreover, we prove a convergence result which shows that the blow-up time can be approximately arbitrarily well if the time steps are scaled sufficiently small. 

The concepts and technical tools developed in our current work constitute an important stepping stone with regard to the numerical treatment of finite time blow-up problems in the context of nonlinear parabolic partial differential equations.

\subsection*{Outline} Our article is organized as follows: Section~\ref{disc} presents the cG and dG time stepping schemes. Furthermore, Section~\ref{existdisc} centres on the development of existence proofs for discrete solutions. The question of uniqueness is addressed in Section~\ref{sc:uniqueness}. Moreover, the application of our results to algebraically growing nonlinearities causing finite time blow-ups will be worked out in Section~\ref{sc:appl}. Finally, the article closes with a few concluding remarks in Section~\ref{conclusions}.

\subsection*{Notation} Throughout the paper, Bochner spaces will be used: For an interval $I=(a,b)$ and a real Hilbert space~$H$ as before, the space $C^0(\overline{I};H)$ consists of all functions $u:\overline I\to H$ that are continuous on~$\overline{I}$ with values in $H$. Moreover, introducing, for~$1\le p\le\infty$, the norm
\[
\|u\|_{L^p(I;H)}=\begin{cases}
\displaystyle\left(\int_I\|u(t)\|^p_H\dif t\right)^{\nicefrac{1}{p}},&1\le p<\infty,\\[2ex]
\text{ess sup}_{t\in I}\|u(t)\|_H,&p=\infty,
\end{cases}
\]
we write $L^p(I;H)$ to signify the space of measurable
functions $u:I\to H$ so that
the corresponding norm is bounded. We notice
that~$L^2(I;H)$ is a Hilbert space with inner product and induced norm
\[
(u,v)_{L^2(I;H)}=\int_I(u(t),v(t))_H\dif t,\qquad\text{and}\qquad \|u\|_{L^2(I;H)}=
\left(\int_I\|u(t)\|^2_H\dif t\right)^{\nicefrac{1}{2}},
\]
respectively.

%%%%%%%%%%%%%%%%%%%%%%%%%%%%%%%%%%%%%%%%%%%%%%%%%%%%%%%%%%%%%%%%%%%%%%
%%%%%%%%%%%%%%%%%%%%%%%%%%%%%%%%%%%%%%%%%%%%%%%%%%%%%%%%%%%%%%%%%%%%%%
\section{Galerkin Time Discretizations}\label{disc}
%%%%%%%%%%%%%%%%%%%%%%%%%%%%%%%%%%%%%%%%%%%%%%%%%%%%%%%%%%%%%%%%%%%%%%
%%%%%%%%%%%%%%%%%%%%%%%%%%%%%%%%%%%%%%%%%%%%%%%%%%%%%%%%%%%%%%%%%%%%%%

In this section we present the $hp$-cG and $hp$-dG time stepping methods as applied to~\eqref{eq:1}.

\subsection{$hp$-cG Time Stepping}
On an interval~$I=[0,T]$, $T>0$, consider time nodes $0 = t_0 < t_1 < \cdots < t_{M-1} < t_M = T$ which introduce a time partition $\M=\{I_m\}_{m=1}^M$ of~$I$ into~$M$ open time intervals~$I_m=(t_{m-1},t_m)$, $m=1,\ldots,M$. The (possibly varying) length $k_m = t_m - t_{m-1}$ of a time interval is called the $m^\text{th}$ time step. Furthermore, to each interval we associate a polynomial degree~$r_m\ge 0$ which takes the role of a local approximation order. Moreover, given a (real) Hilbert space~$X\subset H$, an integer~$r\in\mathbb{N}_0$, and an interval~$J\subset\mathbb{R}$, the set
\[
\mathcal{P}^{r}(J;X)=\left\{p\in C^0(\bar J;X):\,p(t)=\sum_{i=0}^rx_it^i,\, x_i\in X\right\}
\]
signifies the space of all polynomials of degree at most~$r$ on~$J$ with values in~$X$. 

In practical computations, the Hilbert space~$H$, on which~\eqref{eq:1} is based, will typically be replaced by a finite-dimensional subspace~$H_m\subset H$, $\dim(H_m)<\infty$, on each interval~$I_m$, $1\le m\le M$. The $H$-orthogonal projection from~$H$ to~$H_m$ is defined by
\begin{equation}\label{eq:Riesz}
\pi_m:\, H\to H_m,\qquad (x-\pi_mx,y)_H=0\quad\forall y\in H_m.
\end{equation}

With these definitions, the (fully discrete) $hp$-cG time marching scheme is iteratively given as follows: For given initial value~$U_{m-1}:=\lim_{t\nearrow t_{m-1}}U|_{I_{m-1}}(t)\in H$ (with~$U_0:=u_0$, where~$u_0\in H$ is the initial value from~\eqref{eq:1}), we find~$U|_{I_m}\in\mathcal{P}^{r_m+1}(I_m;H_m)$ through the weak formulation
\begin{equation}
\begin{split}
 \int_{I_m} (U',V)_H\dif t&=\int_{I_m} (\F(U),V)_H  \dif t \qquad\forall V\in\mathcal{P}^{r_m}(I_m;H_m),\\
 U(t_{m-1})&=\pi_mU_{m-1},
  \label{eq:cG}
  \end{split}
\end{equation}
for any~$1\le m\le M$. Notice that, in order to enforce the initial condition on each individual time step, the local trial space has one degree of freedom more than the local test space. Furthermore, if~$H_1=H_2=\ldots=H_M$, we remark that the continuous Galerkin solution~$U$ is globally continuous on~$(0,T)$.

%The following existence result for the cG method will be proved in Section~\ref{sc:ex_cG}.

%\begin{theorem}\label{thm:cG}
%Let~$m\ge 1$. Then, if the local time step~$k_m>0$ is chosen sufficiently small (independent of the local polynomial degree~$r_m$), the continuous Galerkin method~\eqref{eq:cG} on~$I_m$ possesses at least one solution~$U\in\mathcal{P}^{r_m+1}(I_m;H_m)$.
%\end{theorem}
%
%\begin{remark}\label{rm:cG}
%It should be emphasized that the existence criterion~\eqref{eq:excG} is independent of the local polynomial degree~$r_m$.
%\end{remark}

\subsection{$hp$-dG Time Stepping} 
In order to define the discontinuous Galerkin scheme, some additional notation is required: We define the one-sided limits of a piecewise continuous function $U$ at each time node $t_m$ by
\[
\Up_m \coloneqq \lim_{s\searrow 0} U(t_m+s), \qquad
\Um_m \coloneqq \lim_{s\searrow 0} U(t_m-s).
\]
Then, the discontinuity jump of $U$ at $t_m$, $0\le m \le M-1$, is defined by $\jump{U}_m := \Up_m - \Um_m$, where we let~$\Um_0\coloneqq u_0$, with~$u_0$ being the initial condition from~\eqref{eq:1}. Then, the (fully discrete) $hp$-dG time stepping method for~\eqref{eq:1} reads: Find $U|_{I_m} \in \mathcal{P}^{r_m}(I_m;H_m)$ such that
\begin{equation}
\begin{split}
  \int_{I_m} (U',V)_H  \dif t 
  &+  (\jump{U}_{m-1}, V_{m-1}^+)_H
   = 
  \int_{I_m}(\F(U),V)_H \dif t\qquad\forall V\in\mathcal{P}^{r_m}(I_m;H_m),
  \label{eq:dG}
  \end{split}
\end{equation}
for any~$1\le m\le M$. We emphasize that, in contrast to the continuous Galerkin formulation, the trial and test spaces are the same for the discontinuous Galerkin scheme. This is due to the fact that the initial values are weakly imposed (by means of an upwind flux) on each time interval.

%
%\begin{remark}
%As for the cG method (cf.~Remark~\ref{rm:cG}), we stress the fact that the existence criterion~%\eqref{eq:exdG} is independent of the local polynomial degree~$r_m$.
%\end{remark}
%

%%%%%%%%%%%%%%%%%%%%%%%%%%%%%%%%%%%%%%%%%%%%%%%%%%%%%%%%%%%%%%%%%%%%%%
%%%%%%%%%%%%%%%%%%%%%%%%%%%%%%%%%%%%%%%%%%%%%%%%%%%%%%%%%%%%%%%%%%%%%%
\section{Existence of Discrete Solutions}\label{existdisc}
%%%%%%%%%%%%%%%%%%%%%%%%%%%%%%%%%%%%%%%%%%%%%%%%%%%%%%%%%%%%%%%%%%%%%%
%%%%%%%%%%%%%%%%%%%%%%%%%%%%%%%%%%%%%%%%%%%%%%%%%%%%%%%%%%%%%%%%%%%%%%

In this Section our goal is to show existence of solutions to the discrete local problems~\eqref{eq:cG} and~\eqref{eq:dG}:

\begin{theorem}\label{thm:cGdG}
Let~$m\ge 1$. Then, if the local time step~$k_m>0$ is chosen sufficiently small (independent of the local polynomial degree~$r_m$), then the continuous Galerkin method~\eqref{eq:cG} and the discontinuous Galerkin method~\eqref{eq:dG} on~$I_m$ both possess at least one solution~$U_{\cG}\in\mathcal{P}^{r_m+1}(I_m;H_m)$ and~$U_{\dG}\in\mathcal{P}^{r_m}(I_m;H_m)$, respectively.
\end{theorem}

Our general strategy of proof is to represent the Galerkin formulations in terms of strong equations, and then to derive suitable fixed-point formulations. Subsequently, the existence of discrete solutions will follow from the application of Brower's fixed point theorem.

\subsection{Existence of cG Solutions}\label{sc:ex_cG}
We begin by rewriting~\eqref{eq:cG} as finding~$U\in\mathcal{P}^{r_m+1}(I_m;H_m)$ such that
\begin{align*}
 \int_{I_m} (U'-\Pi^{r_m}_m\F(U),V)_H\dif t&=0 \qquad\forall V\in\mathcal{P}^{r_m}(I_m;H_m),\\
 U(t_{m-1})&=\pi_mU_{m-1}.
\end{align*}
Here, $\Pi^{r_m}_m:\,L^2(I_m;H)\to \mathcal{P}^{r_m}(I_m;H_m)$ denotes the $L^2$-projection onto the space
$\mathcal{P}^{r_m}(I_m;H_m)$, which is uniquely defined by
\begin{equation}\label{eq:L2}
u\mapsto\Pi_m^{r_m}u:\quad
\int_{I_m}(u-\Pi^{r_m}_mu,V)_{H}\dif t=0\qquad\forall V\in\mathcal{P}^{r_m}(I_m;H_m).
\end{equation}
Thence, noticing that~$U'-\Pi^{r_m}_m\F(U)\in\mathcal{P}^{r_m}(I_m;H_m)$, we obtain the strong form
\begin{align*}
U'-\Pi^{r_m}_m\F(U)&=0\qquad\textrm{on }I_m,\\
U(t_{m-1})&=\pi_mU_{m-1}.
\end{align*}
Integration results in
\begin{equation}\label{eq:inteq}
U(t)=\pi_m U_{m-1}+\int_{t_{m-1}}^t\Pi^{r_m}_m\F(U)\dif\tau,\qquad t\in I_m.
\end{equation}
We see that the operator
\begin{equation}\label{eq:TcG}
\HcG_m(U)(t):=\pi_m U_{m-1}+\int_{t_{m-1}}^t\Pi^{r_m}_m\F(U)\dif\tau
\end{equation}
maps~$\mathcal{P}^{r_m+1}(I_m;H_m)$ into itself, and hence, the integral equation~\eqref{eq:inteq} is a fixed point formulation,
\begin{equation}\label{eq:cGfix}
\HcG_m(U)=U,
\end{equation}
on~$\mathcal{P}^{r_m+1}(I_m;H_m)$. In particular, any solution of~\eqref{eq:inteq} will solve~\eqref{eq:cG}. 

We are now ready to prove Theorem~\ref{thm:cGdG} for the continuous Galerkin method~\eqref{eq:cG}: For some~$\kappa_m, \theta_m>0$ (with~$t_{m-1}+\theta_m\le T$) let us define the set
\[
Q_m=[t_{m-1},t_{m-1}+\theta_m]\times B_m,
\]
where
\begin{equation}\label{eq:cGB}
B_m=\left\{y\in H_m:\,\|y-\pi_mU_{m-1}\|_H\le\kappa_m\right\}.
\end{equation} 
Since~$\F$ is continuous, its maximum on the \emph{compact} set~$Q_m$,
\begin{equation}\label{eq:KcG}
K_m:=\max_{(t,y)\in Q_m}\|\F(t,y)\|_H,
\end{equation}
exists. We let
\begin{equation*}
0<k_m\le\min(\theta_m,K_m^{-1}\kappa_m).
\end{equation*}
Then, we introduce
\begin{equation}\label{eq:McG}
M^{\cG}_m:=\{Y\in\mathcal{P}^{r_m+1}(I_m;H_m):\,Y(t)\in B_m\,\forall t\in\overline I_m\},
\end{equation}
where~$I_m=(t_{m-1},t_m)$, with~$t_m=t_{m-1}+k_m$. 

Let~$U\in M^{\cG}_m$ be arbitrary,
and~$t^\star\in\overline I_m$ such that
\[
\NN{\HcG_m(U)(t^\star)-\pi_mU_{m-1}}_H=\NN{\HcG_m(U)-\pi_mU_{m-1}}_{L^\infty(I_m;H)}.
\]
Then, using Bochner's Theorem as well as the Cauchy-Schwarz inequality, yields
\begin{align*}
\NN{\HcG_m(U)-\pi_mU_{m-1}}_{L^\infty(I_m;H)}
&\le\NN{\int_{t_{m-1}}^{t^\star}\Pi^{r_m}_m\F(U)\dif\tau}_H
\le\int_{I_m}\NN{\Pi^{r_m}_m\F(U)}_H\dif\tau\\
&\le k_m^{\nicefrac{1}{2}}\NN{\Pi^{r_m}_m\F(U)}_{L^2(I_m;H)}.
\end{align*}
Taking into account the boundedness of the~$L^2$-projection on~$I_m$ (with constant~1) leads to
\[
\NN{\HcG_m(U)-\pi_mU_{m-1}}_{L^\infty(I_m;H)}\le k_m^{\nicefrac{1}{2}}\NN{\F(U)}_{L^2(I_m;H)}\le k_m\NN{\F(U)}_{L^\infty(I_m;H)}.
\]
Therefore,
\[
\NN{\HcG_m(U)-\pi_mU_{m-1}}_{L^\infty(I_m;H)}\le K_mk_m\le \kappa_m.
\]
Thus, we have~$\HcG_m(U)\in M^{\cG}_m$, and more generally, it follows~$\HcG_m(M^{\cG}_{m})\subseteq M^{\cG}_{m}$. Finally, since~$M^{\cG}_{m}$ is convex and compact, and~$\HcG_m$ is continuous, Brower's fixed point theorem implies that there exists at least one solution of~\eqref{eq:cGfix} in~$M^{\cG}_{m}$, and thus of~\eqref{eq:cG}.

\subsection{Existence of dG Solutions}

The situation for the dG method is more involved. We will commence by looking at a discrete dG time operator appearing in the dG formulation.

\subsubsection{Discrete dG Time Operator} Following~\cite[Section~4.1]{ScWi10} we define the lifting operator, for~$1\le m\le M$,
\[
\L^{r_m}_m:\,X\to\mathcal{P}^{r_m}(I_m;X),
\]
by
\[
\int_{I_m}(\L^{r_m}_m(z),V)_X\dif t=(z,V^+_{m-1})_X\qquad\forall V\in\mathcal{P}^{r_m}(I_m;X),\, z\in X,
\]
on a real Hilbert space~$X$, with inner product~$(\cdot,\cdot)_X$, and norm~$\|\cdot\|_X$. 

In view of this definition with~$X=H_m$, we have for the dG solution~$U\in\mathcal{P}^{r_m}(I_m;H_m)$ from~\eqref{eq:dG}: 
\begin{align*}
0&=\int_{I_m}\left\{ (U',V)_H-(\F(U),V)_H\right\}\dif t+(\jump{U}_{m-1}, V_{m-1}^+)_H \\
&=\int_{I_m}\left\{ (U',V)_H-(\Pi^{r_m}_m\F(U),V)_H\right\}\dif t
+(\pi_m\jump{U}_{m-1}, V_{m-1}^+)_H \\
%&=\int_{I_m}\left\{ (U'+\L^{r_m}_m(\pi_m\jump{U}_{m-1})-\F(U),V)_H\right\}\dif t\\
&=\int_{I_m}\left\{ (U'+\L^{r_m}_m(\pi_m\jump{U}_{m-1})-\Pi^{r_m}_m\F(U),V)_{H}\right\}\dif t,
\end{align*}
for any~$V\in \mathcal{P}^{r_m}(I_m;H_m)$. Here, $\pi_m$ is the $H$-orthogonal projection from~\eqref{eq:Riesz}, and $\Pi^{r_m}_m$ is the $L^2$-projection from~\eqref{eq:L2}.

Then, since $U'$, $\L^{r_m}_m(\pi_m\jump{U}_{m-1})$, $\Pi^{r_m}_m\F(U)$ all belong to~$\mathcal{P}^{r_m}(I_m;H_m)$, we arrive at the strong formulation
\begin{equation}\label{eq:strong}
U'+\L^{r_m}_m(\pi_m\jump{U}_{m-1})= \Pi^{r_m}_m\F(U)
\end{equation}
of~\eqref{eq:dG}. The term on the left-hand side of this equation is the $hp$-dG time discretization of the continuous derivative operator~$u\mapsto u'$. This motivates the definition of a discrete operator
\begin{equation}\label{eq:chi}
\chi:\,\mathcal{P}^{r_m}(I_m;H_m)\to \mathcal{P}^{r_m}(I_m;H_m)
\end{equation}
given by
\begin{equation}\label{eq:chi2}
U\mapsto \chi(U)=U'+\L^{r_m}_m(U_{m-1}^+).
\end{equation}
For the proof of existence of solutions of~\eqref{eq:strong} it is important to notice that the linear operator~$\chi$ is invertible.

\begin{proposition}\label{pr:chi}
Let~$X$ be a real Hilbert space, and~$1\le m\le M$. Then, the operator~$\chi$ from~\eqref{eq:chi}--\eqref{eq:chi2} is an isomorphism on~$\mathcal{P}^{r_m}(I_m;X)$. In addition, there exists a constant~$C_\chi>0$ independent of the time step~$k_m$ and the local approximation order~$r_m$ such that, for any~$p\in[1,\infty]$, there holds the bound
\begin{equation}\label{eq:chi1}
\NN{U}_{L^\infty(I_m;X)}\le C_\chi k_m^{1-\nicefrac{1}{p}}\NN{\chi(U)}_{L^p(I_m;X)},
\end{equation}
for any~$U\in\mathcal{P}^{r_m}(I_m;X)$. 
\end{proposition}

In order to establish this estimate, we require two auxiliary results which will be proved first.

\begin{lemma}\label{lm:SLC}
Let~$1\le m\le M$, and~$X$ a real Hilbert space. Then, there holds
\begin{equation}\label{eq:auxSLC}
\sup_{t\in I_m}\NN{z-\int_{t_{m-1}}^t\L^{r_m}_m(z)\,\dd\tau}_{X}=\|z\|_X,
\end{equation}
for any~$z\in X$.
\end{lemma}

\begin{proof}
Let us first consider the lifting operator~$\widehat\L^{r_m}:\,X\to\mathcal{P}^{r_m}(\widehat{I};X)$ on the unit interval~$\widehat{I}=(-1,1)$, defined by
\[
\int_{-1}^1(\widehat\L^{r_m}(z),\widehat V)_X\,\dd\hat t=(z,\widehat V(-1))_X\qquad\forall\widehat V\in\mathcal{P}^{r_m}(\widehat I;X),\, z\in X.
\]
Referring to~\cite[Eq.~(35) and Lemma~8]{ScWi10} there holds the explicit formula
\[
z-\int_{-1}^{\hat t}\widehat\L^{r_m}(z)\,\dd\widehat\tau=\frac{z}{2}\left(1-\hat t+\sum_{i=2}^{r_m+1}(-1)^i(2i-1)\widehat Q_i(\hat t)\right),\qquad t\in\widehat I,
\]
where
\[
\widehat Q_i(\hat t)=\int_{-1}^{\hat t}\widehat K_{i-1}(\widehat\tau)\,\dd\widehat\tau=\frac{\widehat K_i(\hat t)-\widehat K_{i-2}(\hat t)}{2i-1},\qquad i\ge 2,
\]
with~$\{\widehat K_i\}_{i\ge 0}$ signifying the family of Legendre polynomials on~$(-1,1)$ (with degrees $\deg(\widehat K_i)=i$), scaled such that~$\widehat K_i(-1)=(-1)^i$; cf.~\cite[Eq.~(9) and Lemma~1]{ScWi10}. Combining the above identities, we obtain
\[
z-\int_{-1}^{\hat t}\widehat\L^{r_m}(z)\,\dd\widehat\tau=\frac{z}{2}\left(1-\hat t+\sum_{i=2}^{r_m+1}(-1)^i\left(\widehat K_i(\hat t)-\widehat K_{i-2}(\hat t)\right)\right).
\]
Noticing the telescope sum as well as the fact that~$\widehat K_0(\hat t)=1$ and~$\widehat K_1(\hat t)=\hat t$, we arrive at
\[
z-\int_{-1}^{\hat t}\widehat\L^{r_m}(z)\,\dd\widehat\tau=\frac{z}{2}(-1)^{r_m+1}\left(\widehat K_{r_m+1}(\hat t)-\widehat K_{r_m}(\hat t)\right).
\]
Then, employing the fact that
\begin{equation}\label{eq:Kinf}
|\widehat K_i(\hat t)|\le 1\qquad\forall \hat t\in[-1,1],\,\forall i\ge 0,
\end{equation}
results in
\[
\NN{z-\int_{-1}^{\hat t}\widehat\L^{r_m}(z)\,\dd\widehat\tau}_X
\le\|z\|_X\qquad\forall \hat t\in\widehat I.
\]
Now we define the affine mapping
\begin{equation}\label{eq:Fm}
F_m:\,\widehat{I}\to I_m,\qquad \hat t\mapsto\frac{1}{2}k_m\hat t+\frac12(t_{m-1}+t_m).
\end{equation}
A scaling argument implies that
\[
\L^{r_m}_m(z)\circ F_m=\frac{2}{k_m}\widehat\L^{r_m}(z);
\]
see~\cite[Lemma~7]{ScWi10}. Hence, by a change of variables, $\tau=F_m(\widehat\tau)$, $\dd\tau=\frac{k_m}{2}\dd\widehat\tau$, we conclude that
\begin{align*}
\NN{z-\int_{t_{m-1}}^{t}\L_m^{r_m}(z)\,\dd\tau}_X
&=\NN{z-\int_{-1}^{F_m^{-1}(t)}\widehat\L^{r_m}(z)\,\dd\widehat\tau}_X
\le\|z\|_X\qquad\forall t\in I_m.
\end{align*}
Noticing that, for~$t=t_{m-1}$, there holds equality in the above bound, completes the proof.
\end{proof}

\begin{lemma}
Let~$1\le m\le M$, and~$X$ a real Hilbert space. Then, the bound
\begin{equation}\label{eq:BN}
\NN{U_{m-1}^+}_X\le \NN{\chi(U)}_{L^1(I_m;X)}
\end{equation}
holds true for any~$U\in\mathcal{P}^{r_m}(I_m;X)$.
\end{lemma}

\begin{proof}
Let~$U\in\mathcal{P}^{r_m}(I_m;X)$. We define
\[
\P:=(-1)^{r_m}U_{m-1}^+\left(\widehat K_{r_m}\circ F_m^{-1}\right)\in\mathcal{P}^{r_m}(I_m; X),
\] 
where~$\widehat{K}_{r_m}$ is the $r_m$-th Legendre polynomial on~$(-1,1)$, which we scale such that~$\widehat K_{r_m}(-1)=(-1)^{r_m}$ (cf. the proof of Lemma~\ref{lm:SLC}), and~$F_m$ is the affine element mapping from~\eqref{eq:Fm}. Then,
\[
\NN{\P}_{L^\infty(I_m;X)}
\le\NN{U_{m-1}^+}_X\NN{\widehat K_{r_m}\circ F_m^{-1}}_{L^\infty(I_m)}
\le\NN{U_{m-1}^+}_X\NN{\widehat K_{r_m}}_{L^\infty(-1,1)}.
\]
Involving~\eqref{eq:Kinf} shows
\begin{equation}\label{eq:P}
\NN{\P}_{L^\infty(I_m;X)}\le \NN{U_{m-1}^+}_X.
\end{equation}
Furthermore, $\P$ is orthogonal to the space~$\mathcal{P}^{r_m-1}(I_m; X)$ (where~$\mathcal{P}^{-1}(I_m; X):=\{0\}\subset X$) with respect to the inner product in~$L^2(I_m;X)$. In particular, since~$U'\in\mathcal{P}^{r_m-1}(I_m;X)$, we have
\[
\int_{I_m}(U',\P)_X\dif t =0.
\]
So, noticing that~$\P(t_{m-1}^+)=U_{m-1}^+$, it follows that
\begin{align*}
\int_{I_m}(\chi(U),\P)_X\dif t
&=\int_{I_m}(\L_m^{r_m}(U^+_{m-1}),\P)_X\dif t\\
&=(U_{m-1}^+,\P(t_{m-1}^+))_X
=\NN{U_{m-1}^+}_X^2.
\end{align*}
Therefore, using H\"older's inequality and recalling~\eqref{eq:P}, we conclude that
\[
\NN{U_{m-1}^+}_X^2\le \NN{\chi(U)}_{L^1(I_m;X)}\NN{\P}_{L^\infty(I_m;X)}\le\NN{\chi(U)}_{L^1(I_m;X)}\NN{U_{m-1}^+}_X.
\]
Dividing by~$\NN{U_{m-1}^+}_X$ shows the desired bound.
\end{proof}

We are now ready to show Proposition~\ref{pr:chi}.

\begin{proof}[Proof of Proposition~\ref{pr:chi}]
%We begin by noticing that, for
Consider ~$U\in\mathcal{P}^{r_m}(I_m;X)$.
%%, we have
%\begin{align*}
%\int_{I_m}(\chi(U),U)_X\dif t
%&=\int_{I_m}(U'+\L_m^{r_m}(U_{m-1}^+),U)_X\dif t\\
%&=\frac12\int_{I_m}\frac{\mathsf{d}}{\mathsf{d}t}\NN{U}^2_X\dif t
%+\NN{U_{m-1}^+}^2_X\\
%&=\frac12\NN{U_{m-1}^+}^2_X+\frac12\NN{U_{m}^-}^2_X,
%\end{align*}
%and therefore
%\begin{align*}
%\NN{U_{m-1}^+}^2_X
%&\le 2\int_{I_m}(\chi(U),U)_X\dif t
%\le2 \int_{I_m}\NN{\chi(U)}_X\NN{U}_X\dif t\\
%&\le 2\NN{U}_{L^\infty(I_m;X)}\NN{\chi(U)}_{L^1(I_m;X)}.
%\end{align*}
%Applying H\"older's inequality, this leads to
%\[
%\NN{U_{m-1}^+}^2_X
%\le 2k_m^{1-\nicefrac{1}{p}}\NN{U}_{L^\infty(I_m;X)}\NN{\chi(U)}_{L^p(I_m;X)},
%\]
%and thus
%\[
%\NN{U_{m-1}^+}_X
%\le \sqrt{2}k_m^{\nicefrac12-\nicefrac{1}{2p}}\NN{U}_{L^\infty(I_m;X)}^{\nicefrac12}\NN{\chi(U)}_{L^p(I_m;X)}^{\nicefrac12}.
%\]
%Moreover, using Young's inequality, $2ab\le\delta a^2+\delta^{-1}b^2$, for any~$a,b\in\mathbb{R}$, $\delta>0$, we see that, for~$a=\sqrt{2}k_m^{\nicefrac12-\nicefrac{1}{2p}}\NN{\chi(U)}_{L^p(I_m;X)}^{\nicefrac12}$ and~$b=\NN{U}_{L^\infty(I_m;X)}^{\nicefrac12}$, it follows:
%\begin{equation}\label{eq:aux1}
%\NN{U_{m-1}^+}_X\le \delta k_m^{1-\nicefrac{1}{p}}\NN{\chi(U)}_{L^p(I_m;X)}+\frac{1}{2\delta}\NN{U}_{L^\infty(I_m;X)}.
%\end{equation}
%
We choose~$t^\star\in\overline I_m$ such that~$\NN{U(t^\star)}_X=\NN{U}_{L^\infty(I_m;X)}$. It  holds that
\[
U(t^\star) = \int_{t_{m-1}}^{t^\star}(U'+\L_m^{r_m}(U^+_{m-1}))\dif\tau
+U_{m-1}^+-\int_{t_{m-1}}^{t^\star}\L_m^{r_m}(U^+_{m-1})\dif\tau.
\]
Applying the triangle inequality as well as Bochner's Theorem, and recalling~\eqref{eq:auxSLC}, this implies that
\begin{align*}
 \NN{U}_{L^\infty(I_m;X)} 
&\le \int_{t_{m-1}}^{t^\star}\NN{\chi(U)}_X\dif\tau
+\NN{U_{m-1}^+-\int_{t_{m-1}}^{t^\star}\L_m^{r_m}(U^+_{m-1})\dif\tau}_X\\
&\le \NN{\chi(U)}_{L^1(I_m;X)}
+\NN{U_{m-1}^+}_X.
\end{align*}
Inserting the bound~\eqref{eq:BN} results in
\[
\NN{U}_{L^\infty(I_m;X)}  \le 2\NN{\chi(U)}_{L^1(I_m;X)},
\]
and applying H\"older's inequality completes the proof with~$C_\chi=2$.
%leads to
%\begin{align*}
%\NN{U}_{L^\infty(I_m;X)}  
%&\le 2k_m^{1-\nicefrac{1}{p}}\NN{\chi(U)}_{L^p(I_m;X)}
%%+\NN{U_{m-1}^+}_X.
%\end{align*}
%Inserting~\eqref{eq:aux1} yields
%\[
%\NN{U}_{L^\infty(I_m;X)} 
%\le k_m^{1-\nicefrac{1}{p}}(1+\delta)\NN{\chi(U)}_{L^p(I_m;X)}
%+\frac{1}{2\delta}\NN{U}_{L^\infty(I_m;X)},
%\]
%which can be transformed to (provided that~$\delta>\frac{1}{2}$)
%\[
%\NN{U}_{L^\infty(I_m;X)} 
%\le k_m^{1-\nicefrac{1}{p}}\eta(\delta)\NN{\chi(U)}_{L^p(I_m;X)},
%\]
%where $\eta(\delta):=(1-\frac{1}{2\delta})^{-1}(1+\delta)>0$,
%for~$\delta>\nicefrac12$. The function~$\eta$ has a local minimum at~$\delta^\star=\nicefrac12\left(1+\sqrt3\right)>\nicefrac12$, with value~$\eta(\delta^\star)=2+\sqrt{3}$. Hence, choosing~$\delta=\delta^\star$, this shows the bound~\eqref{eq:chi1} with~$C_\chi=2+\sqrt{3}$.
%It follows that~$\chi(U)\equiv 0$ if and only if~$U\equiv 0$ on~$I_m$, and therefore, $\chi$ is bijective.
\end{proof}

\begin{remark}
The proof of Proposition~\ref{pr:chi} reveals the upper bound~$C_\chi\le 2$. We emphasize, in particular, that the estimate~\eqref{eq:chi1} is uniform with respect to the local polynomial degree~$r_m\ge 0$ as~$r_m\to\infty$. 
\end{remark}

\begin{remark}
Upon setting~$U=\chi^{-1}(V)$ in~\eqref{eq:chi1}, we obtain
\begin{equation}\label{eq:chiinv}
\|\chi^{-1}(V)\|_{L^\infty(I_m;X)}
\le C_\chi k_m^{1-\nicefrac{1}{p}}\|V\|_{L^p(I_m;X)},
\end{equation}
for any~$V\in\mathcal{P}^{r_m}(I_m;X)$.
\end{remark}

\subsubsection{Fixed Point Formulation and Existence of Discrete dG Solutions}\label{sc:ex_dG}
As for the cG method we prove the existence of solutions of~\eqref{eq:dG} by means of a fixed point argument. For this purpose, we will derive a suitable fixed point formulation, and return to the case~$X=H_m$. Noticing the fact that~$\pi_m U_{m-1}^+=U_{m-1}^+\in H_m$, we observe that, on~$I_m$, there holds
\begin{align*}
U'+\L_m^{r_m}(\pi_m\jump{U}_{m-1})
&=(U-\pi_mU^-_{m-1})'+\L_m^{r_m}(U_{m-1}^+-\pi_mU^-_{m-1})\\
&=\chi(U-\pi_mU^-_{m-1}),
\end{align*}
and recalling~\eqref{eq:strong}, we can write
\[
\chi(U-\pi_mU^-_{m-1})= \Pi^{r_m}_m\F(U).
\]
Applying Proposition~\ref{pr:chi} we infer that
\[
U=\pi_mU^-_{m-1}+\chi^{-1}\left(\Pi^{r_m}_m\F(U)\right);
\]
this is the `dG-version' of the integral equation~\eqref{eq:inteq} for the cG method. Now, for given~$U_{m-1}^-$ (where as before~$U_0^-:=u_0$) we define the operator
\[
\H_m:\,\mathcal{P}^{r_m}(I_m;H_m)\to \mathcal{P}^{r_m}(I_m;H_m)
\]
by
\begin{equation}\label{eq:TdG}
\H_m(U):=\pi_mU_{m-1}^-+\chi^{-1}\left(\Pi^{r_m}_m\F(U)\right).
\end{equation}
Then, $U\in\mathcal{P}^{r_m}(I_m;H_m)$ solves~\eqref{eq:strong} if and only if~$U$ satisfies
\begin{equation}\label{eq:dGFP}
\H_m(U)=U.
\end{equation}

We will now prove the existence of solutions to the local $hp$-dG time stepping scheme~\eqref{eq:dG}: Consider~$\kappa_m, \theta_m>0$ (with~$t_{m-1}+\theta_m\le T$), and define the set
\[
Q_m=[t_{m-1},t_{m-1}+\theta_m]\times B_m,
\]
where
\begin{equation}\label{eq:dGB}
B_m=\left\{y\in H_m:\,\|y-\pi_mU^-_{m-1}\|_H\le\kappa_m\right\}. 
\end{equation}
Due to the continuity of~$\F$, its maximum on the \emph{compact} set~$Q_m$,
\begin{equation}\label{eq:KdG}
K_m:=\max_{(t,y)\in Q_m}\|\F(t,y)\|_H,
\end{equation}
exists. We choose
\begin{equation}\label{eq:kmdG}
0<k_m\le\min(\theta_m,C_\chi^{-1}K_m^{-1}\kappa_m),
\end{equation}
where~$C_\chi$ is the constant from~\eqref{eq:chi1},
and introduce
\begin{equation}\label{eq:MdG}
M^{\dG}_m:=\{Y\in\mathcal{P}^{r_m}(I_m;H_m):\,Y(t)\in B_{m}\,\forall t\in\overline I_m\},
\end{equation}
with~$I_m=(t_{m-1},t_m)$, $t_m=t_{m-1}+k_m$. 

Consider any~$U\in M^{\dG}_m$. From the definition of~$\H_m$ in~\eqref{eq:TdG}, and from~\eqref{eq:chiinv} with~$p=2$, we conclude that
\begin{align*}
\NN{\H_m(U)-\pi_mU_{m-1}^-}_{L^\infty(I_m;H)}
&\le C_\chi k_m^{\nicefrac{1}{2}}\NN{\Pi^{r_m}_m\F(U)}_{L^2(I_m;H)}.
\end{align*}
The boundedness of the $L^2$-projection on~$I_m$ (with constant~1) implies that
\begin{align*}
\NN{\H_m(U)-\pi_mU_{m-1}^-}_{L^\infty(I_m;H)}
\le C_\chi k_m^{\nicefrac{1}{2}}\NN{\F(U)}_{L^2(I_m;H)}.
\end{align*}
Then, we obtain
\begin{align*}
\NN{\H_m(U)-\pi_mU_{m-1}^-}_{L^\infty(I_m;H)}
\le C_\chi k_m\NN{\F(U)}_{L^\infty(I_m;H)}
&\le K_mC_\chi k_m\le\kappa_m,
\end{align*}
since~$U\in M^{\dG}_m$. This implies that~$\H_m(M^{\dG}_m)\subseteq M^{\dG}_m$. Then, employing Brower's fixed point theorem (based on the fact that~$M^{\dG}_m$ is convex and compact, and that~$\H_m$ is continuous), there exists a solution of~\eqref{eq:dGFP}, and therefore of~\eqref{eq:strong} and~\eqref{eq:dG}.

%%%%%%%%%%%%%%%%%%%%%%%%%%%%%%%%%%%%%%%%%%%%%%%%%%%%%%%%%%%%%%%%%%%%%%
%%%%%%%%%%%%%%%%%%%%%%%%%%%%%%%%%%%%%%%%%%%%%%%%%%%%%%%%%%%%%%%%%%%%%%

\section{Uniqueness of Galerkin Solutions}\label{sc:uniqueness}

In order to obtain unique Galerkin solutions on each time step we apply a contraction argument following the approach presented in~\cite{BandleBrunner:94}. To this end, we make the assumption that the nonlinearity~$\F$ is \emph{locally Lipschitz continuous}. Then, if the local time step~$k_m$ in the Galerkin time discretizations is chosen sufficiently small (again, independently of the local polynomial degree), we will show that the operators~$\HcG_m$ and~$\H_m$ from~\eqref{eq:TcG} and~\eqref{eq:TdG}, respectively, are contractive. This will lead to the following uniqueness result.

\begin{theorem}\label{thm:uniqueness}
Let $m\ge 1$, and $\kappa_m,\theta_m>0$ (with~$t_{m-1}+\theta_m\le T$). Furthermore, consider~$B_m$ from~\eqref{eq:cGB}, $K_m$ from~\eqref{eq:KcG}, and~$M_m^{\cG}$ from~\eqref{eq:McG} for the cG method~\eqref{eq:cG}, and the respective quantities for the dG scheme~\eqref{eq:dG} from~\eqref{eq:dGB}, \eqref{eq:KdG}, and~\eqref{eq:MdG}. Moreover, for each of the two schemes, we suppose that there exists a constant~$0\le L_{\F}(B_m)<\infty$ such that the local Lipschitz continuity condition,
\begin{equation}\label{eq:L}
\NN{\F(t,u)-\F(t,v)}_H\le L_{\F}(B_m)\NN{u-v}_H\qquad\forall t\in\overline{I}_m,\,\forall u,v\in B_m,
\end{equation}
holds. In addition, for a parameter~$\varrho\in(0,1)$, suppose that
\begin{equation}\label{eq:kunique}
k_m\le\min\left(\theta_m,c^{-1}K_m^{-1}\kappa_m,\varrho c^{-1}L_{\F}(B_m)^{-1}\right),
\end{equation}
where
\begin{equation}\label{eq:c}
c=\begin{cases}
1&\text{for cG time stepping},\\
C_\chi&\text{for dG time stepping},
\end{cases}
\end{equation}
with~$C_\chi$ being the constant from~\eqref{eq:chi1}. Then, the cG and dG methods on~$I_m$ each possess unique solutions~$U_{\cG}$ and~$U_{\dG}$ in~$M_m^{\cG}$ and~$M_m^{\dG}$, respectively.
\end{theorem}

\begin{proof} 
We treat the cG and dG cases separately.

\emph{Uniqueness of cG solution:} From Section~\ref{sc:ex_cG} we recall the following fact: For given~$\kappa_m,\theta_m>0$ (with~$t_{m-1}+\theta_m\le T$), and for~$K_m$ from~\eqref{eq:KcG}, choosing the local time step~$k_m$ to be bounded by~$k_m\le\min(\theta_m,K_m^{-1}\kappa_m)$ guarantees the self-mapping property~$\HcG_m(M^{\cG}_m)\subseteq M^{\cG}_m$, where~$\HcG_m$ is the cG operator from~\eqref{eq:TcG}. Furthermore, for~$U_1,U_2\in M^{\cG}_m$ we have
\begin{align*}
\big\|\HcG_m(U_1)-\HcG_m(U_2)\big\|_{L^\infty(I_m;H)}
&=\NN{\int_{t_{m-1}}^t\Pi^{r_m}_m\left(\F(U_1)-\F(U_2)\right)\dd\tau}_{L^\infty(I_m;H)}\\
&\le\int_{I_m}\NN{\Pi^{r_m}_m\left(\F(U_1)-\F(U_2)\right)}_{H}\dd\tau\\
&\le k_m^{\nicefrac12}\NN{\Pi^{r_m}_m\left(\F(U_1)-\F(U_2)\right)}_{L^2(I_m;H)}\\
&\le k_m^{\nicefrac12}\NN{\F(U_1)-\F(U_2)}_{L^2(I_m;H)}\\
&\le k_m\NN{\F(U_1)-\F(U_2)}_{L^\infty(I_m;H)}.
\end{align*}
Now involving the Lipschitz condition~\eqref{eq:L} on~$B_m$ from~\eqref{eq:cGB}, we infer that
\[
\NN{\HcG_m(U_1)-\HcG_m(U_2)}_{L^\infty(I_m;H)}
\le L_{\F}(B_m)k_m\NN{U_1-U_2}_{L^\infty(I_m;H)},
\]
for all $U_1,U_2\in M^{\cG}_m$. This implies that, for~$k_m< L_{\F}(B_m)^{-1}$, the operator~$\HcG_m$ is a contraction on~$M^{\cG}_m$. Thus, by the Banach fixed point theorem, the equation~\eqref{eq:cGfix} has a unique solution in~$M^{\cG}_m$.

\emph{Uniqueness of dG solution:} In the case of the dG scheme we proceed in a similar way as for the cG time stepping method. For~$\kappa_m,\theta_m>0$ (with~$t_{m-1}+\theta_m\le T$), and for~$K_m$ from~\eqref{eq:KdG}, choosing the local time step~$k_m$ to be bounded by~$k_m\le\min(\theta_m,C_\chi^{-1}K_m^{-1}\kappa_m)$ ensures that $\H_m(M^{\dG}_m)\subseteq M^{\dG}_m$, where~$C_\chi$ is the constant from~\eqref{eq:chi1}, and~$\H_m$ is the operator defined in~\eqref{eq:TdG}; cf.~Section~\ref{sc:ex_dG}. In addition, for~$U_1,U_2\in M^{\dG}_m$ there holds that
\begin{align*}
\NN{\H_m(U_1)-\H_m(U_2)}_{L^\infty(I_m;H)}
&=\NN{\chi^{-1}\left(\Pi^{r_m}_m\F(U_1)-\Pi^{r_m}_m\F(U_2)\right)}_{L^\infty(I_m;H)}.
\end{align*}
Using~\eqref{eq:chiinv}, we deduce that
\begin{align*}
\NN{\H_m(U_1)-\H_m(U_2)}_{L^\infty(I_m;H)}
&\le C_\chi k_m^{\nicefrac12}\NN{\Pi^{r_m}_m\left(\F(U_1)-\F(U_2)\right)}_{L^2(I_m;H)}\\
&\le C_\chi k_m^{\nicefrac12}\NN{\F(U_1)-\F(U_2)}_{L^2(I_m;H)}\\
&\le C_\chi k_m\NN{\F(U_1)-\F(U_2)}_{L^\infty(I_m;H)}.
\end{align*}
By means of~\eqref{eq:L} we derive the bound
\[
\NN{\H_m(U_1)-\H_m(U_2)}_{L^\infty(I_m;H)}
\le C_\chi L_{\F}(B_m)k_m\NN{U_1-U_2}_{L^\infty(I_m;H)},
\]
for all $U_1,U_2\in M^{\dG}_m$, where the ball~$B_m$ is defined in~\eqref{eq:dGB}. Hence, for $k_m<C_\chi^{-1} L_{\F}(B_m)^{-1}$, the mapping~$\H_m:\, M^{\dG}_m\to M^{\dG}_m$ is a contraction. This implies that the equation~\eqref{eq:dGFP} has a unique solution~$U\in M^{\dG}_m$.
\end{proof}

\begin{remark}\label{rm:FP}
The above Theorem~\ref{thm:uniqueness} shows that the cG and dG operators in~\eqref{eq:TcG} and~\eqref{eq:TdG} are contractions in each time step, and thus, have unique fixed points in~$M_m^{\cG}$ and~$M_m^{\dG}$, respectively. In particular, the corresponding fixed point iterations converge. For instance, in the case of the cG time stepping scheme, for~$m\ge 1$, starting from an initial guess~$U^{(0)}\in M^{\cG}_m$ (which can be chosen, for example, to be the constant function~$U^{(0)}(t)=\pi_mU_{m-1}$, $t\in I_m$), the iteration
\[
U^{(\ell+1)}=\HcG_m(U^{(\ell)}),\qquad \ell\ge 1,
\]
will tend to the unique solution~$U|_{I_m}\in M_m^{\cG}$ of~\eqref{eq:cG}. Similarly, for the dG scheme, for $m\ge1$, and an initial guess~$U^{(0)}\in M_m^{\dG}$ (for example, $U^{(0)}(t)=\pi_mU_{m-1}^-$, $t\in I_m$), the iteration
\[
U^{(\ell+1)}=\H_m(U^{(\ell)}),\qquad \ell\ge 1,
\]
converges to the unique solution~$U|_{I_m}\in M_m^{\dG}$ of~\eqref{eq:dG}. 
%In practice, a possible stopping criterion is to terminate the fixed point iteration on the $m$-th time step if
%\begin{equation}\label{eq:tol}
%\norm{(U_{m}^-)^{\ell}-(U_{m}^-)^{\ell-1}}_H\le \tol\cdot\norm{(U_{m}^-)^{\ell}}_H,
%\end{equation}
%where~$\tol>0$ is a prescribed tolerance.
\end{remark}

%%%%%%%%%%%%%%%%%%%%%%%%%%%%%%%%%%%%%%%%%%%%%%%%%%%%%%%%%%%%%%%%%%%%%%
%%%%%%%%%%%%%%%%%%%%%%%%%%%%%%%%%%%%%%%%%%%%%%%%%%%%%%%%%%%%%%%%%%%%%%

\section{Application to Finite-Time Blow-Up Problems}\label{sc:appl}

In this section we will discuss the existence and uniqueness Theorem~\ref{thm:uniqueness} in the context of nonlinearities~$\F$ that grow algebraically with respect to~$u$, with a power larger than~1. We will show that both the exact solution~$u$ of~\eqref{eq:1} as well as the cG and dG solutions blow up in finite time. In addition, we will provide a time step selection algorithm, and prove a convergence result. In order to keep the technical matters within a reasonable scope, we assume that~$H$ is finite dimensional, and that~$H=H_1=H_2=\ldots=H_m=\ldots$ holds for any~$m\ge 1$. 

\subsection{Algebraic Growth Nonlinearities} 
We consider nonlinearities~$\F$ which feature the following {\em algebraic growth condition}: Suppose that there exist constants~$\alpha,\delta>0$, $\beta>1$, and~$\cF\ge 0$ such that
\begin{equation}\label{eq:growth}
\|\F(t,u)\|_H\le \alpha\|u\|_H^\beta\qquad\text{ and }\qquad
(\F(t,u),u)_H\ge\delta\|u\|_H^{1+\beta},
\end{equation}
for all~$u\in H$ which satisfy~$\|u\|_H\ge\cF$, and for any~$t\in[0,\infty)$ (or for any~$t\in[0,T]$, with sufficiently large~$T>0$). We note that such problems exhibit a blow-up in some finite time~$T_\infty<\infty$. Indeed, let~$u$ solve~\eqref{eq:1}, and suppose that~$\|u_0\|_H>\cF$ in~\eqref{eq:1}. Then, under the conditions~\eqref{eq:growth}, it is easy to see that~$\|u(t)\|_H$ is non-decreasing with respect to~$t$, and thus,
\begin{equation}\label{eq:ODE}
\frac{\dd}{\dd t}\|u(t)\|_H^2=2(u'(t),u(t))_H=2(\F(t,u(t)),u(t))_H\ge2\delta\left(\|u(t)\|_H^2\right)^{\nicefrac{(1+\beta)}{2}}.
\end{equation}
Hence, 
\[
\frac{1}{1-\beta}\frac{\dd}{\dd t}\left[\left(\|u(t)\|_H^2\right)^{\nicefrac{(1-\beta)}{2}}\right]\ge\delta.
\]
Integrating from~$0$ to some~$t>0$ shows that
\[
\|u_0\|_H^{1-\beta}-\|u(t)\|_H^{1-\beta}\ge(\beta-1)\delta t,
\]
and therefore,
\[
t\le\frac{\|u_0\|_H^{1-\beta}}{(\beta-1)\delta}=:\overline{T}_\infty.
\]
It follows that $\overline{T}_\infty$ is an upper bound for the blow-up time.

\subsection{Discrete Blow-Up}
Provided that the properties~\eqref{eq:growth} hold true, the goal of this section is to show that the cG and dG time stepping methods yield solutions which blow-up in finite time. To this end, let us assume, in addition to~\eqref{eq:growth}, that the local Lipschitz property
\begin{equation}\label{eq:Lip}
\NN{\F(t,u)-\F(t,v)}_H\le \gamma\max(\NN{u}_H,\NN{v}_H)^{\beta-1}\NN{u-v}_H
\end{equation}
holds true whenever $\|u\|_H,\|v\|_H\ge\cF$, cf.~\eqref{eq:growth}, and for all~$t\in[0,\infty)$ (or for any~$t\in[0,T]$ with sufficiently large~$T>0$), with a uniform constant~$\gamma\ge 0$.

In the following elaborations, the function
\begin{equation}\label{eq:Psi}
\Psi:\,\left[0,\nicefrac{\gamma}{\alpha}\right)\to\mathbb{R},\qquad
\varrho\mapsto\Psi(\varrho)=\frac{\delta(\gamma-\varrho\alpha)^\beta-\varrho\alpha\gamma^\beta}{\gamma-\varrho\alpha}
\end{equation}
will play an important role; here, $\alpha$, $\beta$, $\delta$, and~$\gamma$ are the constants from~\eqref{eq:growth} and~\eqref{eq:Lip}, respectively. We note that~$\Psi$ is decreasing, and that~$\Psi(0)=\delta\gamma^{\beta-1}>0$, and~$\lim_{\varrho\nearrow\nicefrac{\gamma}{\alpha}}\Psi(\varrho)=-\infty$. Hence, by continuity there exists exactly one zero~$\overline\varrho$ of~$\Psi$ in the interval~$\left[0,\nicefrac{\gamma}{\alpha}\right)$.

\begin{proposition}\label{pr:blowup}
Suppose that the conditions~\eqref{eq:growth} and~\eqref{eq:Lip} hold, and that the initial value~$u_0\in H$ from~\eqref{eq:1} satisfies~$\|u_0\|_H>\cF$. Furthermore, let~$\varrho_0$ be a fixed constant with~$0<\varrho_0<\min(1,\overline\varrho)$, where~$\overline\varrho$ is the unique zero of~$\Psi$ from~\eqref{eq:Psi} in~$[0,\nicefrac{\gamma}{\alpha})$. For any given~$\varrho$ with
\begin{equation}\label{eq:rho}
0<\varrho\le\min\left(\varrho_0,\frac{\alpha^{-1}\gamma}{1+\left(1-\cF\|u_0\|_H^{-1}\right)^{-1}}\right), 
\end{equation}
choose the time steps to be
\begin{equation}\label{eq:km1}
k_m(\varrho):=c^{-1}\gamma^{-\beta}\varrho(\gamma-\varrho\alpha)^{\beta-1}\|U_{m-1}^-\|_H^{1-\beta},\qquad m=1,2,3,\ldots,
\end{equation}
where~$U_{m-1}^-$, $m\ge 1$, signifies the left-sided value of the cG or dG solution~$U$ from~\eqref{eq:cG} or~\eqref{eq:dG}, respectively, at the nodal point~$t_{m-1}$ (with~$U_{m-1}^-=U_{m-1}$ for the cG scheme, and~$U_0^-:=u_0$). Then, there holds:
\begin{enumerate}[(i)]
\item For any~$m\ge 1$, the cG and dG solutions resulting from~\eqref{eq:cG} and~\eqref{eq:dG} exist and are unique in $M^{\cG}_m$ from~\eqref{eq:McG} and~$M_m^{\dG}$ from~\eqref{eq:MdG}, respectively, with~$\kappa_m=\varrho\alpha(\gamma-\varrho\alpha)^{-1}\|U_{m-1}^-\|_H$, for any polynomial degree distribution.
\item Both the cG and the dG solutions blow-up at finite times~$\widetilde{T}^{\cG}_\infty(\varrho)$ and~$\widetilde{T}^{\dG}_\infty(\varrho)$, respectively.
\end{enumerate}
The constants~$\alpha,\beta,\delta$, and~$\gamma$ were introduced in~\eqref{eq:growth} and~\eqref{eq:Lip}, respectively, and $c$ is defined in~\eqref{eq:c}.
\end{proposition}

\begin{proof}
We focus on the dG method only; the proof for the cG method can be done verbatim.
Let~$m\ge 1$, and suppose that the dG solution on the first~$m-1$ time steps is well-defined, and that
\begin{equation}\label{eq:in0}
\|U_{m-1}^-\|_H\ge\|u_0\|_H>c_\F\ge 0. 
\end{equation}
Then, with~$\kappa_m=\eta_m\|U_{m-1}^-\|_H$, where
\begin{equation}\label{eq:eta}
\eta_m=\varrho\alpha(\gamma-\varrho\alpha)^{-1},
\end{equation} 
we see by means of~\eqref{eq:rho} that $0<\eta_m\le1-\cF\|u_0\|_H^{-1}$.
Therefore, for any~$y\in B_m:=\{y\in H:\,\|y-U_{m-1}^-\|_H\le\kappa_m\}$, it follows that
\begin{align*}
\|y\|_H
&\ge\|U_{m-1}^-\|_H-\|y-U_{m-1}^-\|_H\ge\|U_{m-1}^-\|_H-\kappa_m\ge(1-\eta_m)\|U_{m-1}^-\|_H\\
&\ge(1-\eta_m)\|u_0\|_H\ge\cF.
\end{align*}
Consequently, in view of the growth condition~\eqref{eq:growth}, there holds
\begin{align*}
K_m:&=\max_{(t,y)\in Q_m}\|\F(t,y)\|_H\\
&\le\alpha\|y\|^\beta_H
\le\alpha\left(\|U_{m-1}^-\|_H+\kappa_m\right)^\beta
=\alpha\left(1+\eta_m\right)^\beta\|U_{m-1}^-\|_H^\beta,
\end{align*}
with~$Q_m=I_m\times B_m$, where~$I_m=[t_{m-1},t_{m-1}+\theta_m]$, and~$\theta_m:=k_m(\varrho)$. Hence,
\[
k_m(\varrho)=c^{-1}\alpha^{-1}\eta_m\left(1+\eta_m\right)^{-\beta}\|U_{m-1}^-\|_H^{1-\beta}\le
c^{-1}K_m^{-1}\kappa_m,
\] 
and revisiting the existence proof in Section~\ref{sc:ex_dG} (in particular, see~\eqref{eq:kmdG}), we infer that there is a dG solution in~$M^{\dG}_m$.  Furthermore, we bound the Lipschitz constant~$L_\F(B_m)$ appearing in~\eqref{eq:L} by means of~\eqref{eq:Lip}: For any~$u,v\in B_m$ we have~$\|u\|_H,\|v\|_H\ge\cF$ as shown before, and $\max(\|u\|_H,\|v\|_H)\le\|U_{m-1}^-\|_H+\kappa_m$. Thus,
\begin{equation}\label{eq:Lip1}
\begin{split}
L_\F(B_m)&\le \gamma(\kappa_m+\|U_{m-1}^-\|_H)^{\beta-1}
\le\gamma(1+\eta_m)^{\beta-1}\|U_{m-1}^-\|_H^{\beta-1},
\end{split}
\end{equation}
which implies that
\[
k_m(\varrho)=\varrho c^{-1}\gamma^{-1}(1+\eta_m)^{1-\beta}\|U_{m-1}^-\|_H^{1-\beta}
\le\varrho c^{-1}L_\F(B_m)^{-1}.
\]
Then, with reference to~\eqref{eq:kunique}, the uniqueness of a dG solution in~$M^{\dG}_m$ follows immediately. 

Next, consider the dG solution~$U|_{I_m}\in\mathcal{P}^{r_m}(I_m;H)$. Using~\eqref{eq:dG} with the constant test function~$V(t)=U_{m-1}^-$, $t\in I_m$, we have that
\[
(U_m^-,U^-_{m-1})_H=\|U_{m-1}^-\|_H^2+\int_{I_m}(\F(U),U_{m-1}^-)_H\dif t.
\]
Recalling~\eqref{eq:in0}, and employing~\eqref{eq:growth}, we obtain
\begin{align*}
\|U_m^-\|&_H\|U_{m-1}^-\|_H\\
%&\ge \frac12\|U_{m-1}^-\|_H^2+\int_{I_m}(\F(U),U_{m-1}^-)_H\dif t\\
&\ge \|U_{m-1}^-\|_H^2+k_m(\F(U_{m-1}^-),U_{m-1}^-)_H
+\int_{I_m}(\F(U)-\F(U_{m-1}^-),U_{m-1}^-)_H\dif t\\
&\ge \|U_{m-1}^-\|_H^2+k_m\delta\|U_{m-1}^-\|^{1+\beta}_H-\|U_{m-1}^-\|_H\int_{I_m}\|\F(U)-\F(U_{m-1}^-)\|_H\dif t.
\end{align*}
Furthermore, dividing by~$\|U_{m-1}^-\|_H>0$, it holds that
\[
\|U_m^-\|_H
\ge \|U_{m-1}^-\|_H+k_m\delta\|U_{m-1}^-\|^{\beta}_H-\int_{I_m}\|\F(U)-\F(U_{m-1}^-)\|_H\dif t.
\]
Reviewing the proof of Theorem~\ref{thm:cGdG} in Section~\ref{sc:ex_dG}, we observe that~$U(t)\in B_m$ for all~$t\in I_m$. Therefore, using the local Lipschitz continuity~\eqref{eq:Lip} with the bound~\eqref{eq:Lip1}, it follows that
\begin{align*}
\|U_m^-\|_H
&\ge \|U_{m-1}^-\|_H+k_m\delta\|U_{m-1}^-\|^{\beta}_H-k_mL_\F(B_m)\kappa_m\\
&\ge \|U_{m-1}^-\|_H+k_m\left(\delta-\gamma\eta_m(1+\eta_m)^{\beta-1}\right)\|U_{m-1}^-\|^{\beta}_H.
\end{align*}
Inserting~\eqref{eq:eta} yields
\[
\|U_m^-\|_H\ge \|U_{m-1}^-\|_H+k_m\left(\delta-\varrho\alpha\gamma^\beta(\gamma-\varrho\alpha)^{-\beta}\right)\|U_{m-1}^-\|^{\beta}_H.
\]
Then, employing~\eqref{eq:km1}, and recalling that~$\Psi$ is monotone decreasing, leads to
\begin{equation}\label{eq:geo}
\|U_m^-\|_H\ge\left(1+c^{-1}\gamma^{-\beta}\varrho\Psi(\varrho)\right)\|U_{m-1}^-\|_H
\ge (1+C_0\varrho)\|U_{m-1}^-\|_H,
\end{equation}
with
\begin{equation}\label{eq:C0}
C_0=c^{-1}\gamma^{-\beta}\Psi(\varrho_0).
\end{equation}

The assumption~\eqref{eq:in0} is trivially valid for~$m=1$. Furthermore, due to~\eqref{eq:geo} we note the fact that~$\|U_{1}^-\|_H\ge\|u_0\|_H$, and, thus, we conclude inductively that the previous derivations are applicable for any~$m\ge 2$. 

Moreover, from~\eqref{eq:geo} we infer that
\begin{equation}\label{eq:geo2}
\|U_{m-1}^-\|_H\ge (1+C_0\varrho)^{m-k}\|U_{k-1}^-\|_H\qquad\forall m\ge k\ge 1,
\end{equation}
which shows that~$\|U_{m-1}^-\|_H\to\infty$ as~$m\to\infty$. In addition, involving~\eqref{eq:km1} it follows, for any~$m\ge i\ge 1$, that
\begin{align*}
t_m
&=t_{i-1}+\sum_{j=i}^{m}k_j(\varrho)
=t_{i-1}+\frac{\varrho(\gamma-\alpha\varrho)^{\beta-1}}{c\gamma^{\beta}}\sum_{j=i}^{m}\|U_{j-1}^-\|_H^{1-\beta}\\
&\le t_{i-1}+\frac{\varrho(\gamma-\alpha\varrho)^{\beta-1}}{c\gamma^{\beta}}\|U_{i-1}^-\|^{1-\beta}_H\sum_{j=i}^{m} (1+C_0\varrho)^{(1-\beta)(j-i)}\\
&\le t_{i-1}+\frac{\varrho(\gamma-\alpha\varrho)^{\beta-1}}{c\gamma^{\beta}}\|U_{i-1}^-\|^{1-\beta}_H\sum_{j=0}^{\infty} (1+C_0\varrho)^{(1-\beta)j}.
\end{align*}
Therefore,
\begin{equation}\label{eq:time}
t_m\le t_{i-1}+\frac{\varrho(\gamma-\alpha\varrho)^{\beta-1}}{c\gamma^{\beta}}\frac{\|U_{i-1}^-\|^{1-\beta}_H}{1-(1+C_0\varrho)^{1-\beta}},\qquad m\ge i\ge 1.
\end{equation}
In particular, for~$i=1$ and~$m\to\infty$, we see that the discrete blow-up time~$\widetilde{T}^{\dG}_\infty(\varrho)$ for the dG method is bounded by
\begin{equation}\label{eq:blowupup}
\widetilde{T}^{\dG}_\infty(\varrho)\le\frac{\varrho(\gamma-\alpha\varrho)^{\beta-1}}{c\gamma^{\beta}}\frac{\|u_0\|^{1-\beta}_H}{1-(1+C_0\varrho)^{1-\beta}}<\infty.
\end{equation}
This concludes the proof.
\end{proof}

\begin{remark}\label{rm:boundedness}
The above proof allows to establish an $L^\infty$ bound on the cG and dG solution, again denoted by~$U$, on~$(0,t_m)$, for~$m\ge 1$. Indeed, for any~$1\le i\le m$, using~\eqref{eq:eta}, we have that
\begin{align*}
\|U\|_{L^\infty(I_i;H)}
&\le\|U_{i-1}^-\|_H+\kappa_i\le(1+\eta_i)\|U_{i-1}^-\|_H\le\varsigma\|U_{i-1}^-\|_H\le\varsigma\|U_{m-1}^-\|_H,
\end{align*}
with~$\varsigma=\gamma(\gamma-\varrho_0\alpha)^{-1}$. Taking the maximum for all~$1\le i\le m$, we conclude that~$\|U\|_{L^\infty((0,t_m);H)}\le \varsigma\|U_{m-1}^-\|_H$.
\end{remark}

\begin{remark}\label{rm:blowuptimes}
We notice that 
\begin{equation}\label{eq:mu}
\mu:=\lim_{\varrho\searrow 0}\frac{\varrho(\gamma-\alpha\varrho)^{\beta-1}}{c\gamma^{\beta}(1-(1+C_0\varrho)^{1-\beta})}
=\frac{1}{c\gamma C_0(\beta-1)}>0
\end{equation} 
in~\eqref{eq:blowupup}. In particular, we see that the discrete blow-up times $\widetilde{T}^{\cG}_\infty(\varrho)$ and~$\widetilde{T}^{\dG}_\infty(\varrho)$ for the cG and dG methods, respectively, are uniformly bounded for any $\varrho$ satisfying~\eqref{eq:rho}.
\end{remark}

\begin{remark}\label{rm:C0C1}
Using~\eqref{eq:growth}, we can show that there exists a constant~$C_1>0$, with~$C_1\ge C_0$ from~\eqref{eq:C0}, such that
\[
\|U_m^-\|_H\le (1+C_1\varrho)\|U_{m-1}^-\|_H,\qquad m\ge 1,
\]
for both the cG and dG solutions. To see this, consider, for instance, the dG solution~$U|_{I_m}\in\mathcal{P}^{r_m}(I_m;H)$. Applying~\eqref{eq:dG} with the constant test function~$V(t)=U_{m}^-$, $t\in I_m$, yields
\[
\|U_m^-\|_H^2=(U_{m-1}^-,U_{m}^-)_H+\int_{I_m}(\F(U),U_{m}^-)_H\dif t.
\]
Then, proceeding as in the proof of Proposition~\ref{pr:blowup}, there holds
\begin{align*}
\|U_m^-\|_H^2
&= (U_{m-1}^-,U_m^-)_H+k_m(\F(U_{m-1}^-),U_{m}^-)_H+\int_{I_m}(\F(U)-\F(U_{m-1}^-),U_{m}^-)_H\dif t\\
&\le \|U_{m-1}^-\|_H\|U_m^-\|_H+k_m\|\F(U_{m-1}^-)\|_H\|U_{m}^-\|_H\\
&\quad+\|U_{m}^-\|_H\int_{I_m}\|\F(U)-\F(U_{m-1}^-)\|_H\dif t.
\end{align*}
Dividing by~$\|U_m^-\|_H$, involving~\eqref{eq:growth}, \eqref{eq:Lip}, and~\eqref{eq:Lip1}, and recalling that~$\|U-U_{m-1}^-\|_{L^\infty(I_m;H)}\le\kappa_m$, we infer
\begin{align*}
\|U_m^-\|_H
&\le \|U_{m-1}^-\|_H+k_m\alpha\|U_{m-1}^-\|^{\beta}_H+k_mL_\F(B_m)\kappa_m\\
&\le \left(1+k_m\left(\alpha+\gamma\eta_m(1+\eta_m)^{\beta-1}\right)\|U_{m-1}^-\|^{\beta-1}_H\right)\|U_{m-1}^-\|_H.
\end{align*}
Inserting~\eqref{eq:km1} and~\eqref{eq:eta} we arrive at
\begin{equation}\label{eq:geo3}
\begin{split}
\|U_m^-\|_H
&\le
\left(1+\frac{\varrho\alpha(\varrho\gamma^\beta+(\gamma-\varrho\alpha)^\beta)}{c\gamma^\beta(\gamma-\varrho\alpha)}\right)\|U_{m-1}^-\|_H
\le(1+C_1\varrho)\|U_{m-1}^-\|_H,
\end{split}
\end{equation}
with
\begin{equation}\label{eq:C1}
C_1=\alpha c^{-1}(\varrho_0+1)(\gamma-\varrho_0\alpha)^{-1}.
\end{equation} 
Proceeding analogously for the cG method, precisely the same bound can be proved. Moreover, in analogy to the derivation of~\eqref{eq:time}, the bounds
\begin{equation}\label{eq:timelow}
\begin{split}
t_m
&\ge t_{i-1}+\frac{\varrho(\gamma-\alpha\varrho)^{\beta-1}}{c\gamma^\beta}\|U_{i-1}^-\|_H^{1-\beta}\sum_{j=i}^m(1+C_1\varrho)^{(1-\beta)(j-i)}\\
&\ge t_{i-1}+\frac{\varrho(\gamma-\alpha\varrho)^{\beta-1}}{c\gamma^\beta}\|U_{i-1}^-\|_H^{1-\beta}\sum_{j=0}^{m-i}(1+C_1\varrho)^{(1-\beta)j},
\end{split}
\end{equation}
for any~$m\ge i\ge 1$, are obtained.
\end{remark}

%\begin{remark}
%From~\eqref{eq:km1} and~\eqref{eq:geo} it follows that
%\[
%\frac{\|U_m^-\|_H^2-\|U_{m-1}^-\|_H^2}{k_m}
%\ge\frac{C_0c\gamma^{\beta}}{(\gamma-\alpha\varrho)^{\beta-1}}\|U_{m-1}^-\|^{\beta+1}_H\ge C_0c\gamma\|U_{m-1}^-\|^{\beta+1}_H\qquad\forall m\ge 1.
%\]
%Now, for~$m\ge 1$, consider the positive, monotone increasing function
%\[
%\phi_m(t)=\frac{1}{\left(\|U_{m-1}^-\|_H^{1-\beta}-\nicefrac12C_0c(\beta-1)\gamma(t-t_{m-1})\right)^{\nicefrac{2}{(\beta-1)}}},\quad t_{m-1}\le t< t_{m-1}+\frac{2\|U_{m-1}^-\|_H^{1-\beta}}{C_0c(\beta-1)\gamma},
%\]
%which satisfies
%\[
%\phi_m'(t)=C_0c\gamma \phi_m(t)^{\nicefrac{(\beta+1)}{2}},\qquad \phi(t_{m-1})=\|U_{m-1}^-\|^2_H.
%\]
%Then, for any~$l\ge m-1$, we claim that~$\|U_{l}^-\|_H\ge\phi(t_l)$. Evidently, this is true for~$l=m-1$. Now, using induction, we suppose that the statement is true for some~$l\ge m-1$. Then
%\begin{align*}
%\|U_{l+1}\|_H^2
%&\ge\|U_{l}\|_H^2+k_mC_0c\gamma\|U_{l}^-\|_H^{\beta+1}\\
%&\ge\phi_m(t_l)+k_mC_0c\gamma\phi_m(t_l)^{\nicefrac{(\beta+1)}{2}}
%\end{align*}
%
%\end{remark}

\subsection{Convergence to Blow-Up Time}

We will now show that the cG and dG time stepping schemes are able to approximate the exact blow-up time as~$\varrho\searrow 0$ in~\eqref{eq:km1}. To this end, we first establish a few auxiliary results.

\begin{lemma}\label{lm:error}
Suppose that the assumptions of Proposition~\ref{pr:blowup} are fulfilled; in particular choose~$\varrho$ as in~\eqref{eq:rho}, and let the time steps~$\{k_m(\varrho)\}_{m\ge1}$ be given by~\eqref{eq:km1}.  In addition to the local Lipschitz property~\eqref{eq:Lip}, suppose that there exists a constant~$L_{\cF}$ such that
\begin{equation}\label{eq:Lipball}
\|\F(u,t)-\F(v,t)\|_H\le L_{\cF}\|u-v\|_H\qquad\forall t\in[0,\infty),
\end{equation}
whenever $\|u\|_H,\|v\|_H\le\cF$. Furthermore, let~$T_0>0$ be fixed with
\[
T_0<\min(T_\infty(\varrho),\widetilde{T}_\infty(\varrho))<\infty,
\] 
where~$T_\infty(\varrho)$ and~$\widetilde{T}_\infty(\varrho)$ are the exact and the discrete blow-up times (i.e., either the cG or the dG blow-up time), respectively. Moreover, define
\[
M(\varrho,T_0):=\sup\left\{m:\,t_m(\varrho)=\sum_{l=1}^mk_l(\varrho)\le T_0\right\}<\infty,
\]
and
\[
\Xi(u,U,T_0):=\max\left(\|u\|_{L^\infty((0,T_0);H)},\|U\|_{L^\infty((0,t_{M(\varrho,T_0)});H)}\right)<\infty,
\]
where~$u$ is the solution of~\eqref{eq:1}, and~$U$ signifies either the cG or the dG solution. Then, there holds the \emph{a priori} error estimate
\begin{equation}\label{eq:uU}
\norm{u-U}_{L^\infty((0,t_M(\varrho,T_0));H)}\le C(T_0,\Xi(u,U,T_0))\sqrt{C_{\bm r}}\sqrt\varrho,
\end{equation}
where $C_{\bm r}=\sup_{1\le m\le M(\varrho,T_0)}\max(3,\ln(r_m))$, and $C(T_0,\Xi(u,U,T_0))>0$ only depends on the time~$T_0$, on~$L_{\cF}$, on~$\Xi(u,U,T_0)$, and on the constants~$c$, $\alpha$, $\beta$, $\cF$, and~$\gamma$ from~\eqref{eq:c}, \eqref{eq:growth}, and~\eqref{eq:Lip}, respectively.
\end{lemma}

\begin{proof}
Let us first suppose that~$\Xi(u,U,T_0)\ge\cF$. Then, the operator~$\F$ is Lipschitz continuous on the annulus~$R_{T_0}:=\{v\in H:\,\cF\le\|v\|_H\le\Xi(u,U,T_0)\}$, with Lipschitz constant~$L(R_{T_0})=\gamma\Xi(u,U,T_0)^{\beta-1}$; cf.~\eqref{eq:Lip}. Furthermore, by~\eqref{eq:Lipball} we know that~$\F$ is Lipschitz continuous on $\{v\in H:\,\|v\|_H\le \cF\}$, with a Lipschitz constant~$L_{\cF}$. Moreover, if~$\|u\|_H<\cF$ and~$\Xi(u,U,T_0)\ge\|v\|_H\ge\cF$, then we choose~$\omega\in[0,1]$ uniquely such that $z_\omega:=(1-\omega)u+\omega v$ satisfies~$\|z_\omega\|_H=\cF$. We deduce that
\begin{align*}
\|\F(t,u)-\F(t,v)\|_H
&\le\|\F(t,u)-\F(t,z_\omega)\|_H+\|\F(t,z_\omega)-\F(t,v)\|_H\\
&\le L_\F\|u-z_\omega\|_H+L(R_{T_0})\|z_\omega-v\|_H\\
&\le \left(\omega L_\F+(1-\omega)L(R_{T_0})\right)\|u-v\|_H\\
&\le \max\left(L_\F,L(R_{T_0})\right)\|u-v\|_H.
\end{align*}
In summary, we conclude that $\F$ is Lipschitz continuous on~$\{v\in H:\,\|v\|_H\le\Xi(u,U,T_0)\}$, with Lipschitz constant $L_{T_0}:=\max(L_\F,L(R_{T_0}))$.
Evidently, due to~\eqref{eq:Lipball}, this still holds when~$\Xi(u,U,T_0)<\cF$.

Now, we introduce the operator
\[
\G:\,[0,T_0]\times H\to H,\qquad x\mapsto \begin{cases}
\F(t,x)& \|x\|_H\le \Xi(u,U,T_0),\\
\F\left(t,\Xi(u,U,T_0)\frac{x}{\|x\|_H}\right)&\|x\|_H> \Xi(u,U,T_0),
\end{cases}
\]
which is \emph{globally} Lipschitz continuous on~$[0,T_0]\times H$ with Lipschitz constant~$L_{T_0}$; see~Lemma~\ref{lm:FLip}. Then, for the cG method, applying~\cite[Theorem~3.1]{Wihler05}, we obtain, for~$l\in\{0,1\}$, that
\begin{equation*}
\norm{(u-U)^{(l)}}_{L^2((0,t_{M(\varrho,T_0)});H)}^2\le C_{\cG}(T_0,L_{T_0})\|u\|_{H^{1}((0,T_0);H)}^2\max_{1\le m\le M(\varrho,T_0)}k_m^{2(1-l)},
\end{equation*}
for a constant~$C_{\cG}>0$. Therefore, choosing a time~$t^\star\in[0,t_{M(\varrho,T_0)}]$ such that $\|u-U\|_{L^\infty((0,t_{M(\varrho,T_0)});H)}=\|u(t^\star)-U(t^\star)\|_H$, and noticing that~$U_0=u_0=u(0)$, we have
\begin{align*}
\|u-U\|_{L^\infty((0,t_{M(\varrho,T_0)});H)}^2&=\int_0^{t^\star}\frac{\dd}{\dif t}\|u-U\|_H^2\dif t
=2\int_0^{t^\star}(u-U,u'-U')_H\dif t\\
&\le 2\|u-U\|_{L^2((0,t_{M(\varrho,T_0)});H)}\|u'-U'\|_{L^2((0,t_{M(\varrho,T_0)});H)}\\
&\le 2C_{\cG}(T_0,L_{T_0})\|u\|_{H^{1}((0,T_0);H)}^2\max_{1\le m\le M(\varrho,T_0)}k_m\\
&\le 2T_0C_{\cG}(T_0,L_{T_0})\|u\|_{W^{1,\infty}((0,T_0);H)}^2\max_{1\le m\le M(\varrho,T_0)}k_m.
\end{align*}
Moreover, for the dG time stepping scheme we employ~\cite[Theorem~3.12]{SchoetzauSchwabDGODE} to infer
\[
\norm{u-U}_{L^\infty((0,t_{M(\varrho,T_0)});H)}^2\le C_{\dG}(T_0,L_{T_0})C_{\bm r}\|u\|_{W^{1,\infty}((0,T_0);H)}^2\max_{1\le m\le M(\varrho,T_0)}k_m,
\]
where~$C_{\dG}>0$ is again a constant. Then, in view of~\eqref{eq:km1} and~\eqref{eq:geo2}, we observe that
\[
k_m(\varrho)\le\frac{c^{-1}\gamma^{-1}\varrho}{\|U_{m-1}\|_H^{\beta-1}}\le\frac{c^{-1}\gamma^{-1}\varrho}{\|u_0\|_H^{\beta-1}}\le\frac{c^{-1}\gamma^{-1}\varrho}{\cF^{\beta-1}}.
\]
Thus,
\begin{align*}
\|u&-U\|^2_{L^\infty((0,t_{M(\varrho,T_0)});H)}\\
&\le \varrho c^{-1}\gamma^{-1}\cF^{1-\beta}\max\left(2T_0C_{\cG}(T_0,L_{T_0}),C_{\dG}(T_0,L_{T_0})\right)C_{\bm r}\|u\|^2_{W^{1,\infty}((0,T_0);H)},
\end{align*}
where~$U$ is either the cG or dG solution. Upon recalling~\eqref{eq:growth}, we conclude that
\begin{align*}
\|u\|_{W^{1,\infty}((0,T_0);H)}
&\le \|u\|_{L^\infty((0,T_0);H)}+\|\F(u)\|_{L^\infty((0,T_0);H)}\\
&\le \Xi(u,U,T_0)+\alpha\Xi(u,U,T_0)^{\beta},
\end{align*}
and the proof is complete.
\end{proof}

\begin{remark}
We note that the error estimate~\eqref{eq:uU} above is not optimal in terms of~$k_m$ and~$r_m$. It is, however, sufficient to establish the blow-up time convergence result in Theorem~\ref{thm:Tinf} below.
\end{remark}

\begin{lemma}\label{lm:Tsup}
Suppose that the assumptions of Proposition~\ref{pr:blowup} hold. Moreover, consider a time~$T_0>0$ with~$T_{\sup}:=\varlimsup_{\varrho\searrow 0}\widetilde T_\infty(\varrho)>T_0$ (note that, by Remark~\ref{rm:blowuptimes}, it holds that~$T_{\sup}<\infty$). Furthermore, let~$\{\varrho_l\}_{l\ge 1}$ be a sequence with~$\varrho_l\xrightarrow{l\to\infty}0^+$ that satisfies the bound~\eqref{eq:rho}, and~$\lim_{l\to\infty}\widetilde T_\infty(\varrho_l)=T_{\sup}$. Moreover, by Remark~\ref{rm:blowuptimes}, we may suppose that the sequence~$\{\varrho_l\}_l$
satisfies
\begin{equation}
\label{eq:murho}
\frac{\varrho_l(\gamma-\alpha\varrho_l)^{\beta-1}}{c\gamma^{\beta}(1-(1+C_0\varrho_l)^{1-\beta})}\le 2\mu\qquad\forall l\ge 1,
\end{equation}
where~$\mu>0$ is the constant from~\eqref{eq:mu}. Then, whenever~$l,m\in\mathbb{N}$ are such that
\begin{equation}\label{eq:tmT0}
t_m(\varrho_l)=\sum_{i=1}^mk_i(\varrho_l)\le T_0,
\end{equation}
with~$k_m$ from~\eqref{eq:km1}, the dG and cG time stepping solutions are bounded by
\[
\|U_m^-(\varrho_l)\|_H\le\left(\frac{T_{\sup}-T_0}{2\mu}\right)^{\nicefrac{1}{(1-\beta)}}.
\]
\end{lemma}

\begin{proof}
Suppose that~\eqref{eq:tmT0} holds. Then, applying~\eqref{eq:time} (with~$m\to\infty$) it follows that
\[
0<T_{\sup}-T_0\le T_{\sup}-t_m(\varrho_l)
\le\frac{\varrho_l(\gamma-\alpha\varrho_l)^{\beta-1}}{c\gamma^{\beta}(1-(1+C_0\varrho_l)^{1-\beta})}\|U_m^-(\varrho_l)\|_H^{1-\beta}.
\]
Using~\eqref{eq:murho}, we infer that~$T_{\sup}-T_0\le 2\mu\|U_m^-(\varrho_l)\|_H^{1-\beta}$, which shows the assertion.
\end{proof}

Before stating the next lemma, we recall, by Remark~\ref{rm:C0C1}, that~$C_0\le C_1$. Hence, for any~$A\in\mathbb{R}$ with~$A>\|u_0\|_H>0$, there holds
\[
A\le\left(\frac{A}{\|u_0\|_H}\right)^{\nicefrac{C_1}{C_0}}\|u_0\|_H
<2\left(\frac{A}{\|u_0\|_H}\right)^{\nicefrac{C_1}{C_0}}\|u_0\|_H.
\]

\begin{lemma}\label{lm:m}
Suppose that the assumptions of Proposition~\ref{pr:blowup} are fulfilled. Furthermore, let~$A>\|u_0\|_H$. Then, there exists a sequence~$\varrho_l\xrightarrow{l\to\infty}0^+$ (satisfying the bound~\eqref{eq:rho}) with
\begin{equation}\label{eq:seq}
\lim_{l\to\infty}\widetilde T_\infty(\varrho_l)=T_{\inf}:=\varliminf_{\varrho\searrow 0}\widetilde T_\infty(\varrho), 
\end{equation}
and a time~$T_0<T_{\inf}$ (depending, in particular, on~$A$ and on~$\|u_0\|_H$) such that for any~$l$ there is a time index~$m_A(\varrho_l)\ge 0$ with
\[
t_{m_A(\varrho_l)}\le T_0,\qquad A\le\|U(\varrho_l)_{m_A(\varrho_l)}^-\|_H\le 2\left(\frac{A}{\|u_0\|_H}\right)^{\nicefrac{C_1}{C_0}}\|u_0\|_H.
\]
Here, we denote by~$U(\varrho_l)$ either the discrete cG or dG solution from~Proposition~\ref{pr:blowup}, and by~$\widetilde T_\infty(\varrho_l)$ the corresponding discrete blow-up time. Furthermore, $C_0$ and~$C_1$ are the constants from~\eqref{eq:C0} and~\eqref{eq:C1}, respectively.
\end{lemma}

\begin{proof}
Due to~\eqref{eq:geo2}, for~$\|U_{m}^-\|_H\ge A$ to hold, it is sufficient that
\[
m\ge m_{A}(\varrho):=\left\lceil\frac{\ln(A\|u_0\|_H^{-1})}{\ln(1+C_0\varrho)}\right\rceil.
\]
Then, using~\eqref{eq:timelow} with~$M\ge m_A(\varrho)+1\ge 1$, we have
\begin{align*}
t_{M}-t_{m_A(\varrho)}
&\ge\frac{\varrho(\gamma-\alpha\varrho)^{\beta-1}}{c\gamma^\beta}\|U_{m_A(\varrho)}^-\|_H^{1-\beta}\sum_{j=0}^{M-m_A(\varrho)-1}(1+C_1\varrho)^{(1-\beta)j}\\
&=\frac{\varrho(\gamma-\alpha\varrho)^{\beta-1}}{c\gamma^\beta}\|U_{m_A(\varrho)}^-\|_H^{1-\beta}\frac{1-(1+C_1\varrho)^{(1-\beta)(M-m_A(\varrho))}}{1-(1+C_1\varrho)^{1-\beta}}.
\end{align*}
Letting~$M\to\infty$, leads to
\[
\widetilde{T}_\infty(\varrho)-t_{m_A(\varrho)}\ge\frac{\varrho(\gamma-\alpha\varrho)^{\beta-1}}{c\gamma^\beta}\frac{\|U_{m_A(\varrho)}^-\|_H^{1-\beta}}{1-(1+C_1\varrho)^{1-\beta}}.
\]
Applying~\eqref{eq:geo3} $m_A(\varrho)$-times, we note that
\[
\norm{U_{m_A(\varrho)}^-}_H\le(1+C_1\varrho)^{m_A(\varrho)}\|u_0\|_H
\le(1+C_1\varrho)^{1+\frac{\ln(A\|u_0\|_H^{-1})}{\ln(1+C_0\varrho)}}\|u_0\|_H.
\]
We observe that
\[
\lim_{\varrho\searrow 0}(1+C_1\varrho)^{1+\frac{\ln(A\|u_0\|_H^{-1})}{\ln(1+C_0\varrho)}}
=(A\|u_0\|_H^{-1})^{\nicefrac{C_1}{C_0}}.
\]
Furthermore,
\[
\widetilde{T}_\infty(\varrho)-t_{m_A(\varrho)}\ge\frac{\varrho(\gamma-\alpha\varrho)^{\beta-1}\|u_0\|_H^{1-\beta}}{c\gamma^\beta\left(1-(1+C_1\varrho)^{1-\beta}\right)}(1+C_1\varrho)^{(1-\beta)\left(1+\frac{\ln(A\|u_0\|_H^{-1})}{\ln(1+C_0\varrho)}\right)},
\]
and since the right-hand side of the above inequality tends to
\[
\nu:=\frac{\|u_0\|_H^{1-\beta}}{c(\beta-1)\gamma C_1}\left(A\|u_0\|_H^{-1}\right)^{\nicefrac{(1-\beta)C_1}{C_0}}>0,
\]
as~$\varrho\searrow 0$, we conclude that we can choose~$\varrho^\star$ small enough (and satisfying~\eqref{eq:rho}) so that
\[
\norm{U_{m_A(\varrho)}^-}_H\le 2(A\|u_0\|_H^{-1})^{\nicefrac{C_1}{C_0}}\|u_0\|_H,\qquad
\widetilde{T}_\infty(\varrho)-t_{m_A(\varrho)}\ge\frac{\nu}{2},
\]
for any~$0<\varrho\le\varrho^\star$. Now consider a sequence~$\varrho_l\xrightarrow{l\to\infty}0^+$, with $0<\varrho_l\le\varrho^\star$ for all~$l$, that satisfies~\eqref{eq:seq} as well as
\[
\left|\widetilde{T}_\infty(\varrho_l)-T_{\inf}\right|\le\frac{\nu}{4}\qquad\forall l.
\]
Then, upon defining $T_0:=T_{\inf}-\nicefrac{\nu}{4}$,
we see that
\begin{align*}
T_0=t_{m_A(\varrho_l)}+\left(\widetilde{T}_\infty(\varrho_l)-t_{m_A(\varrho_l)}\right)-\left(\widetilde{T}_\infty(\varrho_l)-T_{\inf}\right)-\frac{\nu}{4}\ge t_{m_A(\varrho_l)},
\end{align*}
and thus, the proof is complete.
\end{proof}

We are now ready to show the following result on the convergence of the Galerkin time stepping schemes to the exact blow-up time.

\begin{theorem}\label{thm:Tinf}
Let the assumptions of Proposition~\ref{pr:blowup} and of Lemma~\ref{lm:error} be satisfied, and suppose that~$\sup_{m\ge1}r_m<\infty$. Then, there holds
\[
\lim_{\varrho\searrow 0}\widetilde{T}_\infty(\varrho)=T_\infty,
\]
where $\widetilde{T}_\infty(\varrho)$ denotes either the discrete cG or dG blow-up time, and~$T_\infty<\infty$ is the blow-up time of~\eqref{eq:1} under the conditions~\eqref{eq:growth} and~\eqref{eq:Lip}. 
\end{theorem}

\begin{proof}
We establish the proof by contradiction. 

Suppose first that $T_{\sup}:=\varlimsup_{\varrho\searrow0}\widetilde{T}_\infty(\varrho)>T_\infty$. Thence, $\infty>\Delta_\infty:=T_{\sup}-T_\infty>0$. We can find a sequence~$\{\varrho_l\}_{l}\subset\mathbb{R}_{>0}$ satisfying~\eqref{eq:rho}, with~$\varrho_l\xrightarrow{l\to\infty}0^+$, such that, for all~$l$, there holds
\begin{equation}\label{eq:sup}
\widetilde{T}_\infty(\varrho_l)\ge T_\infty+\frac12\Delta_\infty,
\end{equation}
as well as~\eqref{eq:murho}. Since the exact solution~$u$ of~\eqref{eq:1} blows up at~$T_\infty$, there is a time~$T_0$, $T_0< T_\infty$, such that
\[
\|u(T_0)\|_H\ge 2\left(\frac{4\mu}{\Delta_\infty}\right)^{\nicefrac{1}{(\beta-1)}}.
\]
Furthermore, choosing~$l'$ large enough, there exists a time node index $m(\varrho_{l'})$ with
\[
T_0\le t_{m(\varrho_{l'})}\le\frac12(T_0+T_\infty),
\]
and such that~$\sqrt{\varrho_l}$ is sufficiently small. Referring to Lemma~\ref{lm:Tsup}, we have
\[
\|U_{m(\varrho_{l'})}^-\|_H\le\left(\frac{T_{\sup}-\nicefrac12(T_0+T_\infty)}{2\mu}\right)^{\nicefrac{1}{(\beta-1)}},
\]
uniformly with respect to~$\varrho_{l'}$. Hence, by virtue of Remark~\ref{rm:boundedness} and Lemma~\ref{lm:error} (noting that~$C_{\bm r}<\infty$ in~\eqref{eq:uU}), it is possible to establish the estimate
\[
\|u(t_{m(\varrho_{l'})})-U_{m(\varrho_{l'})}^-\|_H\le \left(\frac{4\mu}{\Delta_\infty}\right)^{\nicefrac{1}{(\beta-1)}}.
\]
Since~$t\mapsto\|u(t)\|_H$ is non-decreasing this implies that
\begin{align*}
\|U_{m(\varrho_{l'})}^-\|_H
&\ge\|u(t_{m(\varrho_{l'})})\|_H-\|u(t_{m(\varrho_{l'})})-U_{m(\varrho_{l'})}^-\|_H\\
&\ge\|u(T_0)\|_H-\left(\frac{4\mu}{\Delta_\infty}\right)^{\nicefrac{1}{(\beta-1)}}\ge\left(\frac{4\mu}{\Delta_\infty}\right)^{\nicefrac{1}{(\beta-1)}}.
\end{align*}
Then, recalling~\eqref{eq:time}, leads to
\begin{align*}
\widetilde{T}_\infty(\varrho_{l'})
&\le t_{m(\varrho_{l'})}+\frac{\|U_{m(\varrho_{l'})}^-\|_H^{1-\beta}}{c\gamma^\beta}\frac{\varrho_{l'}(\gamma-\alpha\varrho_{l'})^{\beta-1}}{1-(1+C_0\varrho_{l'})^{1-\beta}}\\
&< T_\infty+2\mu\|U_{m(\varrho_{l'})}^-\|_H^{1-\beta}\le T_\infty+\frac12\Delta_\infty,
\end{align*}
which is a contradiction to~\eqref{eq:sup}.

Next, let us assume that $T_{\inf}:=\varliminf_{\varrho\searrow0}\widetilde{T}_\infty(\varrho)<T_\infty$,
and define $\Delta_\infty:=T_\infty-T_{\inf}>0$. Furthermore, let
\[
A:=\max\left(2\left(\delta(\beta-1)\Delta_\infty\right)^{\nicefrac{1}{(1-\beta)}},1+\|u_0\|_H\right)>\|u_0\|_H.
\]
Due to Lemma~\ref{lm:m} we can find a sequence~$\{\varrho_l\}_{l}\subset\mathbb{R}_{>0}$, with~$\varrho_l\xrightarrow{l\to\infty}0^+$, and a time~$T_0<T_{\inf}$ so that, for all $l$, there exists~$m_A(\varrho_l)$ with $t_{m_A(\varrho_{l})}\le T_0$,
and
\[
A\le\|U_{m_A(\varrho_{l})}^-\|_H\le 2\left(\frac{A}{\|u_0\|_H}\right)^{\nicefrac{C_1}{C_0}}\|u_0\|_H.
\]
In particular, $\|U_{m_A(\varrho_{l})}^-\|_H$ is bounded independently of~$\varrho_l$; evidently, since~$T_0<T_{\inf}<T_\infty$, it follows that~$\|u(t_{m_A(\varrho_l)})\|_H$ is bounded as well. Thus, as before, recalling Remark~\ref{rm:boundedness}, and using Lemma~\ref{lm:error}, we may find a sufficiently large index $l'$ such that
\[
\|u(t_{m_A(\varrho_{l'})})-U_{m_A(\varrho_{l'})}^-\|_H\le \frac12A.
\]
Then, 
\begin{align*}
\|u(t_{m_A(\varrho_{l'})})\|_H
&\ge \|U_{m_A(\varrho_{l'})}^-\|_H-\|u(t_{m_A(\varrho_{l'})})-U_{m_A(\varrho_{l'})}^-\|_H\\
&\ge \frac12A\ge\left(\delta(\beta-1)\Delta_\infty\right)^{\nicefrac{1}{(1-\beta)}}.
\end{align*}
Integrating~\eqref{eq:ODE} from $t_{m_A(\varrho_{l'})}$ to~$T_\infty$, we arrive at
\begin{align*}
T_\infty&\le t_{m_A(\varrho_{l'})}+ \frac{\|u(t_{m_A(\varrho_{l'})})\|_H^{1-\beta}}{\delta(\beta-1)}
\le T_0+\Delta_\infty
<T_{\inf}+\Delta_\infty=T_\infty,
\end{align*}
which constitutes a contradiction.

In summary, we have shown that
\[
\varlimsup_{\varrho\searrow0}\widetilde{T}_\infty(\varrho)\le T_\infty\le\varliminf_{\varrho\searrow0}\widetilde{T}_\infty(\varrho),
\]
which concludes the proof.
\end{proof}

\subsection{A Time Step Selection Algorithm}
The theory in the previous sections suggests the following algorithm for computing a numerical approximation of the exact blow-up time of~\eqref{eq:1} under the conditions~\eqref{eq:growth} and~\eqref{eq:Lip}.

\begin{algorithm}\label{alg:Tinf}
Suppose that the assumptions of Proposition~\ref{pr:blowup} are satisfied. Choose a parameter~$\varrho$ as in~\eqref{eq:rho}, and a tolerance~$\tau>0$. Then:
\begin{algorithmic}[1]
   	\State Set $m=0$; $\widetilde{T}_\infty=0$;
	\Loop
      \State $m \gets m+1$;
      \State Compute $k_m(\varrho)$ using~\eqref{eq:km1};
      \State $\widetilde{T}_\infty\gets\widetilde{T}_\infty+k_m(\varrho)$;
      \If {$k_m(\varrho)>\tau$}
      	\State{Compute the cG~\eqref{eq:cG} or the dG~\eqref{eq:dG} solution on~$I_m$}
		\State{(in~$M_m^{\cG}$ from~\eqref{eq:McG} or~$M_m^{\dG}$ from~\eqref{eq:MdG}, respectively);}
      	\Else\State{\textbf{return} $\widetilde{T}_\infty$};
	  \EndIf
   \EndLoop
\end{algorithmic}
The result, $\widetilde{T}_\infty$, is an approximation of the exact blow-up time.
\end{algorithm}

\begin{remark}\label{rm:until}
A practical (and platform-independent) implementation of the stopping criterion in the {\tt if}-statement in line~6 of the above Algorithm~\ref{alg:Tinf} is to run the time marching process until 
\[
\widetilde{T}_\infty+k_m(\varrho)\,\verb+==+\,\widetilde{T}_\infty
\] 
is true, where we make use of the equality operator ``\verb+==+''. Note that this also eliminates the need of specifying the tolerance parameter~$\tau$.
\end{remark}

In order to provide an illustrating example for Algorithm~\ref{alg:Tinf}, let us consider the initial value problem of finding a function~$u=u(t)$, $t\ge 0$, such that
\[
u'(t)=\frac{(|u(t)|+1)u(t)}{1+e^{-t}}=:\F(t,u(t)),\qquad u(0)=3.
\] 
It has an exact solution $u(t)=3(e^t+1)(5-3e^t)^{-1}$,
and, thus, features a blow-up at~$T_\infty=\ln(\nicefrac{5}{3})$. Here, $H=\mathbb{R}$, and~$\cF<3$ in~\eqref{eq:growth} in alignment with Proposition~\ref{pr:blowup}. Choosing~$\cF=2$, a few elementary calculations show that $\alpha=\nicefrac{3}{2}$, $\beta=2$, and~$\delta=\nicefrac{1}{2}$ in~\eqref{eq:growth}. Furthermore, $\gamma=\nicefrac{5}{2}$ in~\eqref{eq:Lip}, and~\eqref{eq:Lipball} holds with~$L_\F=5$. The unique positive root of~$\Psi$ in~\eqref{eq:Psi} is given by~$\overline\varrho\approx0.243163$. We see that the assumptions of Theorem~\ref{thm:Tinf} hold, although, in our computations, we select larger values of~$\varrho$ and of~$k_m(\varrho)$ than would be mandated by~\eqref{eq:rho} and~\eqref{eq:km1}, respectively. More precisely, we consider
\[
k_m(\varrho)=\varrho|U_{m-1}^-|^{-1},
\]
for~$\varrho\in\{2^{-\nicefrac{p}{2}},p=4,\ldots,10\}$. We run Algorithm~\ref{alg:Tinf} based on the stopping criterion mentioned in Remark~\ref{rm:until}. The polynomial degree~$r=r_m$ is kept fixed for all time steps. The results are displayed in Figure~\ref{fig:cGdG} (left) for both the cG and the dG time stepping methods for different values of~$\varrho$, and for various choices of~$r\in\{0,\ldots,3\}$. The numerical solution on each time step is obtained with the aid of a fixed point iteration as described in Remark~\ref{rm:FP}
%, with a tolerance~$\mathtt{tol}=10^{-14}$ 
(we note that solving the nonlinear problems by means of an adaptive Newton method is potentially more efficient from a computational view point, however, we remark that this approach seems more fragile close to the blow-up due to the large magnitude of the numerical solution). Even though Theorem~\ref{thm:Tinf} does not provide any theoretical evidence on convergence rates, the results suggest a convergence to blow-up time of order~$\mathcal{O}(\varrho^{2(r+1)})$ for the cG method (based on a local polynomial degree~$r+1$), and~$\mathcal{O}(\varrho^{2r+1})$ for the dG scheme (based on a local polynomial degree~$r$), for~$r=0,1,2$; for~$r\ge3$ the errors are too small to allow for a precise identification of the convergence behaviour. The number of time steps was found to be independent of the polynomial degrees (however, strongly dependent on~$\varrho$) and of whether the cG or dG method was employed; see Figure~\ref{fig:cGdG} (right).
\begin{figure}
\centering
\includegraphics[width=0.48\linewidth]{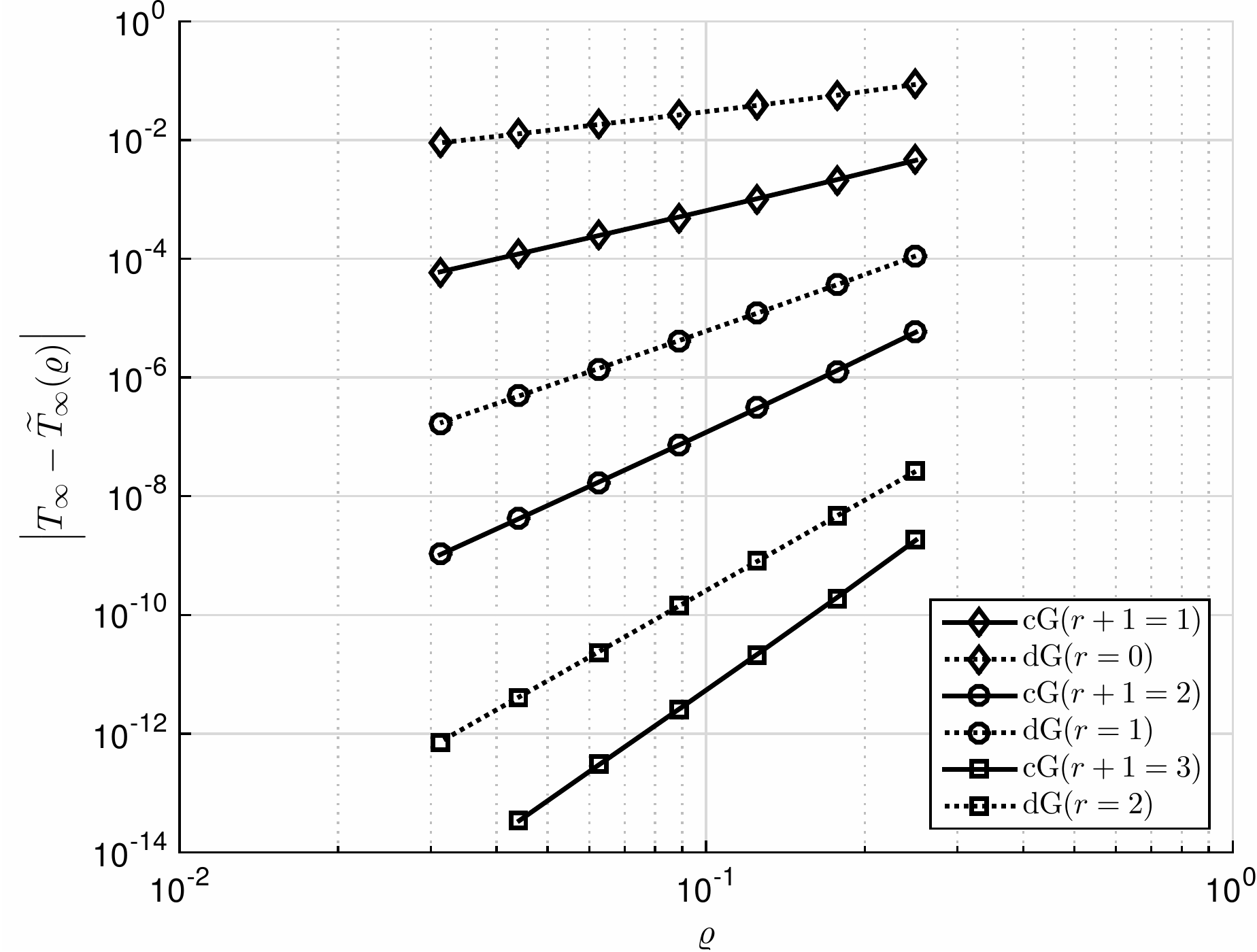}\hspace{3ex}
\includegraphics[width=0.47\linewidth]{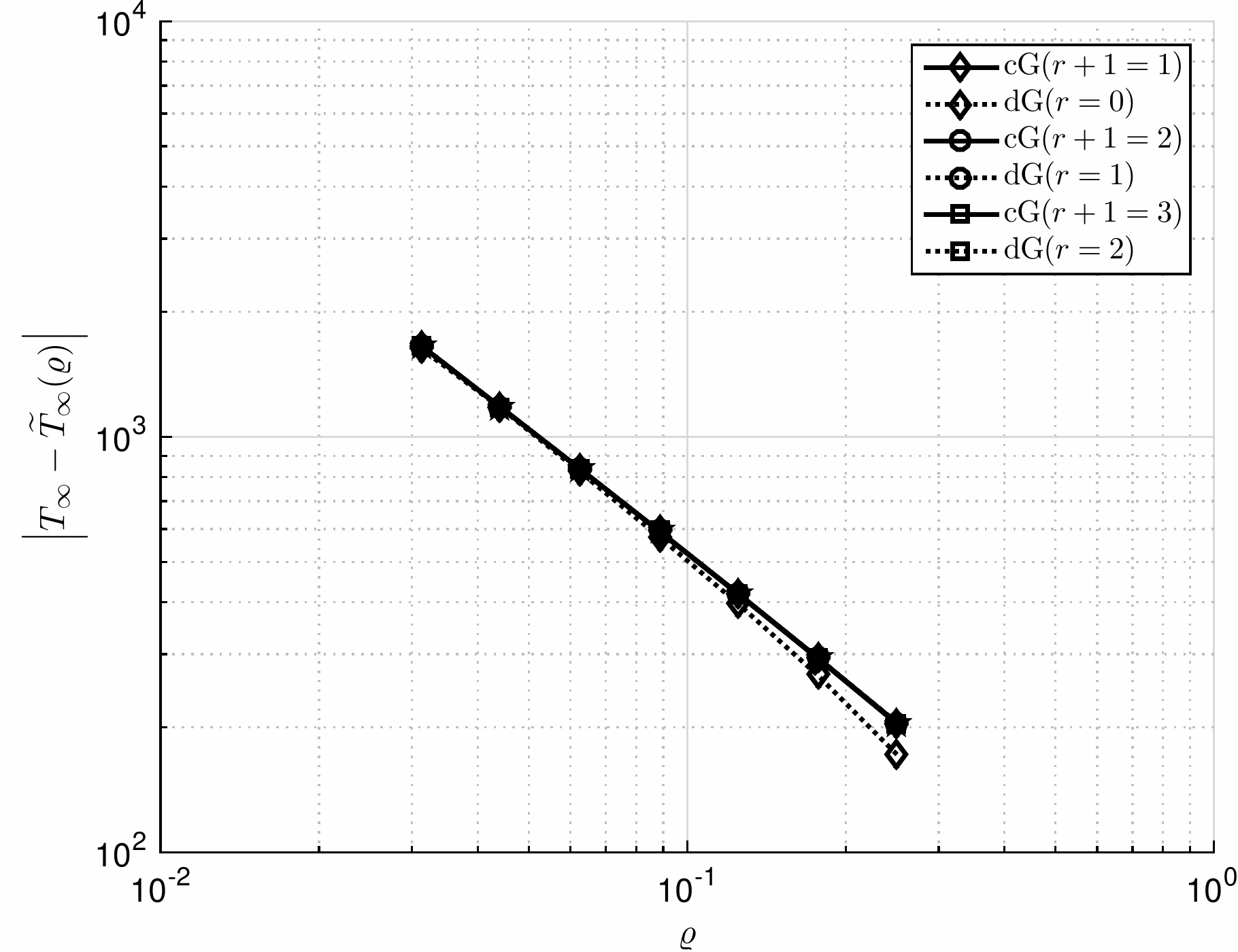}
\caption{Errors of approximation of blow-up time for the cG and dG time stepping schemes for different choices of polynomial degrees (left), and corresponding number of time steps (right).}
\label{fig:cGdG}
\end{figure}

%%%%%%%%%%%%%%%%%%%%%%%%%%%%%%%%%%%%%%%%%%%%%%%%%%%%%%%%%%%%%%%%%%%%%%
%%%%%%%%%%%%%%%%%%%%%%%%%%%%%%%%%%%%%%%%%%%%%%%%%%%%%%%%%%%%%%%%%%%%%%
\section{Conclusions}\label{conclusions}

In this paper we have investigated the $hp$-version continuous and discontinuous Galerkin time stepping methods for the numerical approximation of general initial value problems with continuous (and possibly unbounded) nonlinearities in real Hilbert spaces. Our main findings include Peano-type existence results for the discrete systems, and a blow-up time step selection algorithm, together with a convergence result, for problems with algebraically growing nonlinearities. We have shown that discrete solutions exist (and are unique within suitable ranges) provided that the local time steps are chosen sufficiently small (depending on the numerical solutions themselves, however, {\em independent} of the local polynomial degrees). The key ingredients in the existence and uniqueness proofs include the derivation of strong forms of the Galerkin discretizations, the transformation into suitable fixed point equations, and the application of fixed point theory. The application of the techniques derived in this article to nonlinear parabolic partial differential equations, and the development of \emph{a posteriori} error estimates for the blow-up time (in conjunction with the $hp$-framework) are subjects of ongoing research.

%%%%%%%%%%%%%%%%%%%%%%%%%%%%%%%%%%%%%%%%%%%%%%%%%%%%%%%%%%%%%%%%%%%%
%%%%%%%%%%%%%%%%%%%%%%%%%%%%%%%%%%%%%%%%%%%%%%%%%%%%%%%%%%%%%%%%%%%%%

\appendix
\section{An Auxiliary Result}

\begin{lemma}\label{lm:FLip}
Let~$\F:\,H\to H$ be a continuous function on a (real) Hilbert space~$H$, and~$M>0$ a constant such that the Lipschitz condition
\begin{equation}\label{eq:FLip}
\|\F(x)-\F(y)\|_H\le L_M\|x-y\|_H
\end{equation}
holds for any~$x,y\in H$ with~$\|x\|_H,\|y\|_H\le M$; here~$L_M>0$ is a constant. Then, the function
\[
\G:\,H\to H,\qquad x\mapsto \begin{cases}
\F(x)& \|x\|_H\le M,\\
\F\left(\frac{Mx}{\|x\|_H}\right)&\|x\|_H> M
\end{cases}
\]
is (globally) Lipschitz continuous on~$H$ with Lipschitz constant~$L_M$.
\end{lemma}

\begin{proof}
For $\|x\|_H,\|y\|_H\le M$ the claim follows immediately from the definition of~$\G$ and from~\eqref{eq:FLip}. If~$\|x\|_H,\|y\|_H> M$, we have
\begin{align*}
\|\G(x)-\G(y)\|_H\le L_MM\norm{\frac{x}{\|x\|_H}-\frac{y}{\|y\|_H}}_H,
\end{align*}
where we notice that
\begin{align*}
\norm{\frac{x}{\|x\|_H}-\frac{y}{\|y\|_H}}_H^2
&=\frac{2}{\|x\|_H\|y\|_H}\left(\|x\|_H\|y\|_H-(x,y)_H\right)\\
&\le \frac{1}{\|x\|_H\|y\|_H}\left(\|x\|_H^2+\|y\|^2_H-2(x,y)_H\right)
<\frac{1}{M^2}\|x-y\|_H^2.
\end{align*}
Thus, $\|\G(x)-\G(y)\|_H< L_M\|x-y\|_H$. Moreover, if~$\|x\|_H\le M<\|y\|_H$, then it holds that
\[
\|\G(x)-\G(y)\|_H\le L_MM\norm{\frac{x}{M}-\frac{y}{\|y\|_H}}_H,
\]
where
\begin{align*}
\norm{\frac{x}{M}-\frac{y}{\|y\|_H}}_H^2
&=\frac{1}{M\|y\|_H}\|x-y\|_H^2
-\frac{1}{M\|y\|_H}\left(\frac{\|y\|_H}{M}-1\right)\left(\|y\|_HM-\|x\|_H^2\right) \\
&<\frac{1}{M\|y\|_H}\|x-y\|_H^2<\frac{1}{M^2}\|x-y\|_H^2.
\end{align*}
Therefore, again~$\|\G(x)-\G(y)\|_H< L_M\|x-y\|_H$. The proof for~$\|x\|_H> M\ge\|y\|_H$ follows from symmetry.
\end{proof}

%%%%%%%%%%%%%%%%%%%%%%%%%%%%%%%%%%%%%%%%%%%%%%%%%%%%%%%%%%%%%%%%%%%%
%%%%%%%%%%%%%%%%%%%%%%%%%%%%%%%%%%%%%%%%%%%%%%%%%%%%%%%%%%%%%%%%%%%%%

\bibliographystyle{plain}
\bibliography{paper}
\end{document}